\documentclass{article}

\usepackage[utf8]{inputenc}
\usepackage[english]{babel}
\usepackage{amsfonts}
\usepackage[11pt]{moresize}
\usepackage{pgfplots}
\usepackage{xfrac}
\usepackage{amssymb}
\usepackage{mathtools}
\usepackage{tikz}
\usepackage{tikz-cd}
\usepackage{mathtools}
\usepackage{amsmath}
\usepackage{amsthm}
\usepackage{hyperref}
\usepackage{float}
\usetikzlibrary{calc, positioning,decorations.markings,arrows}

\usepackage{graphics}
\usepackage{color}
\usepackage{tabularx,colortbl}
\usetikzlibrary{positioning}
\usetikzlibrary{decorations.markings,arrows}
\usetikzlibrary{calc}
\usepackage{amsthm}

\theoremstyle{plain}
\newtheorem{trm}{Theorem}[section]
\newtheorem{lm}[trm]{Lemma}
\newtheorem{prop}[trm]{Proposition}
\newtheorem{cor}[trm]{Corollary}

\theoremstyle{definition}
\newtheorem{defi}[trm]{Definition}
\newtheorem{rmk}[trm]{Remark}

\def\OO{\mathcal{O}}

\def\cA{\mathcal{A}}

\def\cM{\mathcal{M}}
\def\cR{\mathcal{R}}
\def\rr{\overline{\mathcal{R}}}

\def\Pic0{{\rm Pic}^0(X)}

\begin{document}
\title{Prym-Brill-Noether theory for ramified double covers}
\author{Andrei Bud}
\date{}
\maketitle
\begin{abstract}
	We initiate the study of Prym-Brill-Noether theory for ramified double covers, extending several key results from classical Prym-Brill-Noether theory to this new framework. In particular, we improve Kanev's results on the dimension of pointed Prym-Brill-Noether loci for ramified double covers. Additionally, we compute the dimension of twisted Prym-Brill-Noether loci with vanishing conditions at points, thus extending the results of Tarasca. Furthermore, we compute the class of the twisted Prym-Brill-Noether loci inside (a translation of) the Prym variety, thus extending the results of de Concini and Pragacz to ramified double covers. Finally, we prove that a generic Du Val curve is Prym-Brill-Noether general.   
\end{abstract}
 
\section{Introduction}
Starting with the fundamental work of Mumford and Beauville, see \cite{MumfordPrym} and \cite{Beau77}, the moduli space of Prym curves
\[ \cR_g \coloneqq \left\{[C,\eta] \ | \ [C]\in \cM_g, \ \eta^{\otimes2}\cong\OO_C\  \mathrm{and} \ \eta \neq \OO_C\right\} \] 
became an important object of study in Algebraic Geometry. A key motivation for this study is the Prym map 
\[ \mathcal{P}_g\colon \cR_g\rightarrow \cA_{g-1} \]
which links the geometry of curves to the geometry of principally polarized Abelian varieties.




In recent years, several variations of this moduli space of Prym curves have been studied, including pointed Prym curves, see \cite{tarasca-pointed-prym}, and ramified Prym curves, see \cite{Lelli-Chiesa-uniruled}, \cite{Lelli-Chiesa-lowgenus} and \cite{BudKodPrym}. Thus it is natural to consider the moduli spaces of ramified Prym curves
\[ \mathcal{R}_{g,2k} \coloneqq \left\{[C,\eta, B] \ | \ [C] \in \cM_g, \ \eta\in \textrm{Pic}^k(C) \ \textrm{and} \ B \ \textrm{is a reduced divisor in} \ |\eta^{\otimes 2} |  \right\} \]
and study their geometry. The equivalence between such tuples $[C,\eta, B]$ and double covers $f\colon \widetilde{C} \rightarrow C$ with branch divisor $B$ will be particularly useful in this study. In many instances, we will think of the moduli space $\mathcal{R}_{g,2k}$ as parametrizing double covers ramified at $2k$ points instead.
 
Moreover, the moduli spaces $\cR_{g,2k}$ appear naturally as boundary loci in the compactification $\rr_g$ of $\cR_{g}$, see \cite{FarLud} and \cite{Casa}. As such, Brill-Noether questions on $\cR_g$ can be reduced via degeneration to questions on these moduli of ramified Prym curves. One advantage of this is that the elements in $\cR_{g,2}$, $\cR_{g,4}$ and $\cR_{g,6}$ are better behaved from the perspective of Brill-Noether theory than their counterpart in $\cR_g$, see \cite{BudKodPrym} and \cite{Lelli-Chiesa-lowgenus}. 
 
Motivated by the many applications of Brill-Noether Theory in understanding the birational geometry of $\cM_{g,n}$ (cf. \cite{KodMg}, \cite{KodevenHarris1984}, \cite{FarPayneJensen} and the references therein), as well as the role of Prym-Brill-Noether Theory in studying the geometry of $\cR_g$ (cf. \cite{FarVerNikulin} and \cite{BudPBN}), our aim is to understand the geometry of Prym-Brill-Noether loci of a double cover $f\colon \widetilde{C} \rightarrow C$ corresponding to a generic element $[C,\eta, B]$ of $\cR_{g,2k}$. Given such a cover $f\colon \widetilde{C} \rightarrow C$, the Prym-Brill-Noether locus, for some $r\geq 0$, is defined as the closed set
\[ V^r(f) = \overline{\left\{L \in \textrm{Pic}^{2g-2}(\widetilde{C}) \ | \ \textrm{Nm}_f(L)= \omega_C, \  \ h^0(\widetilde{C}, L) = r+1 \ \right\}}.\]
Similarly, the twisted Prym-Brill-Noether locus is defined as
\[ V^r_\eta(f) \coloneqq \left\{L\in \mathrm{Pic}^{2g-2+k}(\widetilde{C})\ | \ \textrm{Nm}_f(L) = \omega_C\otimes \eta, \ \textrm{and} \ h^0(\widetilde{C}, L) \geq r+1 \right\}. \]
The goal of this paper is to study the geometry of Prym-Brill-Noether loci. Specifically, we will compute their dimension as well as their class inside the singular cohomology. Moreover, once we have the characterization of these loci as degeneracy loci of expected dimension, we immediately obtain a description of their singular locus. Our results improve on the existing literature in multiple ways. The dimension of these Prym-Brill-Noether loci was estimated from bellow in \cite{Kanev}; while the class of these loci inside the singular cohomology was computed in \cite{DeConciniPragacz} only in the unramified case $[f\colon \widetilde{C} \rightarrow C] \in \cR_g$.  

\subsection{Moduli spaces of ramified curves and Brill-Noether conditions} Starting with a ramified double cover $f\colon \widetilde{C} \rightarrow C$, generic in the moduli space $\cR_{g,2k}$, it is natural to study the geometric properties of both curves and of the map itself. 

A study of the geometric properties of $\widetilde{C}$ already appeared in the literature,  with satisfying results when the number of ramification points is low (i.e. $2, 4$ or $6$). This was motivated by the study of the birational geometry of $\cR_{g,2k}$ in both low genera, see \cite{Lelli-Chiesa-lowgenus}, \cite{Lelli-Chiesa-uniruled}, and high genus cases, see \cite{BudKodPrym}. In these cases, it is known that $\widetilde{C}$ is Brill-Noether general, see \cite{BudKodPrym} and \cite{Lelli-Chiesa-lowgenus}. This property is crucial for finding effective divisors on $\mathcal{R}_{g,2k}$: under good numerical assumptions, the locus of covers $f\colon \widetilde{C} \rightarrow C$ with $\widetilde{C}$ not Brill-Noether general is an effective divisor in $\rr_{g,2k}$ and its class can be explicitly computed.
 
One of the objectives of this paper is to study the behaviour of the curve $\widetilde{C}$ with respect to Gieseker-Petri conditions. In Theorem \ref{trm: coupled k = 1} and Theorem \ref{trm: coupled k = 2} we obtain results about the Gieseker-Petri generality of $\widetilde{C}$, when the double cover $f$ has $2$ or $4$ ramification points:
\begin{trm} Let $f\colon \widetilde{C}\rightarrow C$ be a generic double cover in $\cR_{g,2}$ or $\cR_{g,4}$. Then the curve $\widetilde{C}$ is Gieseker-Petri general. 
\end{trm}

These results are obtained by deforming the double cover $f\colon \widetilde{C} \rightarrow C$ to the boundary of $\cR_{g,2k}$. The compactified space $\rr_{g,2k}$ was described in \cite[Section 2]{BudKodPrym} in terms of line bundles on the base, and can be alternatively described in terms of admissible double covers, see \cite{AbramovichCV}.

Since Gieseker-Petri conditions describe the tangent space for both twisted Prym-Brill-Noether loci, as well as their pointed counterparts, we obtain several results about the dimension and the singular locus of these spaces in Section \ref{sec: twist}.

\subsection{Prym-Brill-Noether loci for ramified double covers}

Similar to the case of $\cR_g$, there is a natural way to construct Abelian varieties out of the datum of a ramified double cover, thereby relating the geometry of pointed curves to the geometry of Abelian varieties. Moreover, when $k = 1$, these Abelian varieties can be equipped with a principal polarization, giving a map
\[\mathcal{P}_{g,2}\colon \cR_{g,2}\rightarrow \cA_g.\]
that allows the study of principally polarized Abelian varieties via the geometry of curves. The image of this map sits as an intermediary between Jacobian varieties and Prym varieties, i.e. 
\[ \overline{J}_g \subset \overline{\textrm{Im}(\mathcal{P}_{g,2})} \subset \overline{P}_{g+1} \coloneqq \overline{\textrm{Im}(\mathcal{P}_{g+1})}.  \]
This further motivates the study of $\cR_{g,2}$ from the perspective of the Schottky problem, which aims to identify explicit geometric, algebraic or analytic conditions that distinguish Jacobians and Prym varieties from generic Abelian varieties in $\cA_g$. While the literature on the Schottky problem for Jacobians and Prym varieties is extensive, see \cite{Schottky-Andreotti-Mayer}, \cite{MumfordPrym}, \cite{Beauville-Schottky} among others, little is known about Abelian varieties in the image of $\mathcal{P}_{g,2}$.

For a generic Prym variety, the singularities of its theta divisor are described through the Prym-Brill-Noether Theory of its corresponding Prym curve in $\cR_{g}$. This suggests a similar phenomenon for Abelian varieties in the image of $\mathcal{P}_{g,2}$, further motivating the study of Prym-Brill-Noether conditions for ramified Prym curves. 

When $f\colon \widetilde{C} \rightarrow C$ is a double cover in $\cR_{g,2k}$, Kanev proved that 
\[ V^r(f) = \overline{\left\{L \in \textrm{Pic}^{2g-2}(\widetilde{C}) \ | \ \textrm{Nm}_f(L)= \omega_C, \  \ h^0(\widetilde{C}, L) = r+1 \ \right\}}.\]
has dimension at least $g-1 + k -k(r+1) - \frac{r(r+1)}{2}$, see \cite{Kanev}. When $k = 1$, we also prove the reverse inequality for the dimension and obtain:
    \begin{trm} \label{theoremPBN-dim}
	Let $r\geq 0$ and $f\colon \widetilde{C} \rightarrow C$ a generic double cover in $\mathcal{R}_{g,2}$. Then the Prym-Brill-Noether locus $V^r(f)$ has dimension $g-\frac{(r+1)(r+2)}{2}$ and is smooth away from the locus of line bundles with strictly more than $r+1$ independent global sections.   
\end{trm}

The smoothness in the theorem is an immediate consequence of our proof. The main idea is to use the description of the tangent space provided in \cite[Section 1]{Kanev} and prove it has dimension $g - \frac{(r+1)(r+2)}{2}$ by degenerating to the boundary. 

Through the map $\iota\colon \rr_{g,2} \rightarrow \Delta_0^{\textrm{ram}} \subseteq \rr_{g+1}$, this Prym-Brill-Noether locus is related to degenerations of Prym-Brill-Noether loci in $\rr_{g+1}$. As shown in \cite{Welters}, the tangent space of the Prym-Brill-Noether loci can be described in terms of a pointed Prym-Gieseker-Petri map. The injectivity of this map will imply Theorem \ref{theoremPBN-dim}. 

  \subsection{Twisted Prym-Brill-Noether loci}
  Another motivation for considering ramified Prym curves arises from the study of vector bundles. Via the BNR correspondence, the study of double covers is related to the study of rank $2$ vector bundles on the target curve, see \cite{BNR}. In fact, starting with a line bundle $L$ satisfying $\textrm{Nm}_f(L) = \omega_C \otimes \eta$, we obtain, via pushforward, a vector bundle $E = f_*L$ satisfying $\det(E) = \omega_C$. Via this pushforward, we obtain a sublocus of the space of stable rank $2$ vector bundles with canonical determinant.
  
  In this paper, we study the twisted Prym-Brill-Noether locus $V^r_\eta(f)$ for double covers $f$ with $2k\leq 4$ ramification points. We extend the results of \cite{Kanev} and \cite{DeConciniPragacz} to ramified double covers. Our approach adapts the methods of de Concini and Pragacz, see \cite{DeConciniPragacz} to describe the dimension, singular locus and class of the twisted Prym-Brill-Noether loci. 
  
  By viewing the twisted Prym-Brill-Noether loci as subspaces of the Prym variety 
  \[ P \coloneqq \left\{L\in \mathrm{Pic}^{2g-2+k}(\widetilde{C})\ | \ \textrm{Nm}_f(L) = \omega_C\otimes \eta,  \right\},\] 
  we can compute their class in the numerical equivalence ring $N^*(P, \mathbb{C})$ or the singular cohomology  $H^*(P, \mathbb{C})$. This class is expressed in terms of the class $\theta'$, the restriction to $P$ of the theta divisor on $\textrm{Pic}^{2g-2+k}(\widetilde{C})$. 
     
	\begin{trm} \label{maintwistrm} Let $f\colon \widetilde{C}\rightarrow C$ a generic element in $\cR_{g,2k}$ for $0\leq k \leq 2$ and let $\eta$ be the torsion line bundle defining the double cover. Then the twisted Prym-Brill-Noether locus 
		\[ V^r_\eta(f)  \coloneqq \left\{L \in \mathrm{Pic}^{2g-2+k}(\widetilde{C}) \ | \ \textrm{Nm}_f(L) = \omega_C\otimes \eta \ \textrm{and} \ h^0(\widetilde{C}, L)\geq r+1\right\} \]
		has dimension 
		\[g+k-1-\frac{(r+1)(r+2)}{2}.\]
		For $k=1$ and $k=2$, this locus is smooth away from $V^{r+1}_\eta(f)$. 
		
		Moreover, the class of the twisted Prym-Brill-Noether locus inside $N^*(P, \mathbb{C})$, or $H^*(P, \mathbb{C})$, is given by the formula
		\[ [V^r_\eta(f)] = \prod_{i=1}^{r+1}\frac{i!}{(2i)!}\cdot (\theta')^{\frac{(r+1)(r+2)}{2}}.\]
	\end{trm}

As proved in Section \ref{sec: BN-GP}, when $f\colon \widetilde{C} \rightarrow C$ is very general in $\cR_{g,2}$ or $\cR_{g,4}$, the curve $\widetilde{C}$ is Gieseker-Petri general. This generality ensures the smoothness of the twisted Prym-Brill-Noether locus away from the sublocus of line bundles with more than expected independent global sections.

In fact, we prove in Section \ref{sec: BN-GP} that the source curve with a generic point on it satisfies the coupled Gieseker-Petri condition. Following the approach in \cite{tarasca-pointed-prym}, this can be used to obtain a pointed version of the previous result.

\begin{trm}
	Let $f \colon \widetilde{C} \rightarrow C$ be a generic element of $\cR_{g,2k}$ for $k = 1$ or $2$, and let $\eta$ be the torsion line bundle defining the cover. Moreover, let $p \in \widetilde{C}$ be a generic point and let $\textbf{a} = (0\leq a_0 <  \cdots< a_r\leq 2g-2+k)$ be a vanishing sequence. Then the locus 
\[
\mathrm{V}_\eta^{\textbf{a}}(f, p) := \left\{ L\in \mathrm{Pic}^{2g-2+k}(\widetilde{C}) \, \left\vert \, 
\begin{array}{l}
	\mathrm{Nm}(L)= \omega_C\otimes \eta,  \\[0.2cm]
	h^0\text{\normalfont\large(}\widetilde{C}, L(-a_i\,p)\text{\normalfont\large)}\geq r+1-i, \,\, \forall \, 0\leq i \leq r 
\end{array}
\right.
\right\}.
\]
has dimension 
\[g + k - r- 2 - \sum_{i=0}^{r} a_i.\] 
Moreover, its class inside $N^*(P, \mathbb{C})$, or $H^*(P, \mathbb{C})$, is given by the formula 
\[  [\mathrm{V}_\eta^{\textbf{a}}(f, p)] = \prod_{i=0}^r \frac{1}{(a_i+1)!}\prod_{0\leq j < i \leq r} \frac{a_i-a_j}{a_i+a_j+2} (\theta')^{|\textbf{a}| + r +1}.\]
\end{trm}

Our proof of this result is simple and straight-forward: By using the coupled Gieseker-Petri condition we immediately compute the dimension of the tangent space for a generic element in the locus. This gives the inequality
\[ \dim V^{\textbf{a}}_\eta(f,p) \leq g - 1 + k -r-1-\sum_{i=0}^ra_i.\] 
To obtain the converse inequality, we give a description of $V^{\textbf{a}}_\eta(f,p)$ as a Lagrangian degeneracy locus, see Subsection \ref{sub: Lagrangian} and Section \ref{sec: twist}.

\subsection{Boundary degenerations and Prym-Brill-Noether general curves} 

In many contexts within Algebraic Geometry, the geometric behaviour over the smooth locus can be understood in terms of combinatorial properties at the boundary. This observation raises the natural question: What limit linear series appear as degenerations of Prym-Brill-Noether loci? 

We consider the universal Prym-Brill-Noether locus 
\[ \mathcal{V}^r_g \coloneqq \left\{ [f\colon \widetilde{C}\rightarrow C, L] \ | \ [f] \in \cR_g \ \textrm{and} \ L \in V^r(f)\right\}. \]
along with the forgetful map 
\[ \mathcal{V}^r_g \rightarrow \cR_g. \]
Furthermore, we consider the diagram 
\[
\begin{tikzcd}
	\mathcal{V}^r_g  \arrow{r}{} \arrow[swap]{d}{} & ? \arrow{d}{} \\
	\cR_g \arrow{r}{i}&  \rr_g
\end{tikzcd}
\]
and ask for a suitable compactification of $\mathcal{V}^r_g$ that completes it. 

Understanding the compactification of this space led to several important results in Prym-Brill-Noether Theory. For instance, degenerating $C$ to a chain of elliptic curves was instrumental in proving the injectivity of the Prym-Gieseker-Petri map and in computing the dimension of the generic fiber, see \cite{Welters}. More recently, when $g - \frac{r(r+1)}{2}$ equals $0$ or $1$, this degeneration was used to compute the class of the Prym-Brill-Noether divisor, see \cite{BudPBN} as well as to show the irreducibility of the Universal Prym-Brill-Noether locus, see \cite{BudPrymIrr}. 

Viewing the space $V^r(f)$ as a subspace of $W^r_{2g-2}(\widetilde{C})$, it is natural to consider limit linear series $g^r_{2g-2}$ satisfying the norm condition and ask which ones appear when degenerating $\mathcal{V}^r_g$ to the boundary. Our focus will be on how this universal Prym-Brill-Noether locus degenerate above the boundary divisor $\Delta_1$, see Proposition \ref{Prym limits}. This degeneration provides a link between Prym-Brill-Noether loci and pointed Brill-Noether loci, which are more thoroughly understood. Using this link, we provide an alternative proof that for a generic $[f\colon \widetilde{C}\rightarrow C] \in \cR_g$ the dimension of $V^r(f)$ is $g-1- \frac{(r+1)r}{2}$. 

In \cite{Farkas-Tarasca-BNgen}, Farkas and Tarasca find curves in $\cM_{g,1}$ that are pointed Brill-Noether general. Using the link between Prym-Brill-Noether conditions and pointed Brill-Noether conditions, we can adapt their methods and find specific curves in every genus $g$ that are Prym-Brill-Noether general (i.e. all Prym-Brill-Noether loci are either empty or of expected dimension). 

We begin this discussion by recalling the setting described in \cite{Farkas-BNgeneral-curves}: Let $S$ be the blow-up of $\mathbb{P}^2$ at nine points $p_1, \ldots, p_9$ which are general in the sense of \cite{Farkas-BNgeneral-curves}, and denote $E_1, \ldots, E_9$ the exceptional divisors. We consider the linear system 
\[ L_g \coloneqq |3gl-gE_1-\cdots-gE_8-(g-1)E_9|\]
where $l$ is the proper transform of a line in $\mathbb{P}^2$. 

A Du Val curve is defined as a genus $g$ curve $C$ in the linear system $L_g$. A generic Du Val curve is Brill-Noether general, see \cite{Farkas-Tarasca-BNgen} and \cite{Farkas-BNgeneral-curves}. The same method as in these papers can be used to find curves that are Prym-Brill-Noether general:  

\begin{trm} \label{trm: DuVal} Let $C$ be a general Du Val curve in $\cM_g$. Then there exists a $2$-torsion line bundle $\eta$ on $C$ such that $[C,\eta]$ is Prym-Brill-Noether general. 
\end{trm}

The same approach can be employed to find Prym-Brill-Noether general curves on decomposable ruled surfaces and $K3$ surfaces, see Section \ref{sec: surfaces}. 

\begin{rmk}
	It is straightforward to produce examples of nine general points in $\mathbb{P}^2$ in the sense of \cite{Farkas-BNgeneral-curves} starting from a concrete elliptic curve. For instance, it is shown in \cite{Farkas-BNgeneral-curves} that the following points lying on the elliptic curve $E = Z(y^2-x^3-17)$ are general: $p_1 = (-2,3), p_2 = (-1,-4), p_3 = (2,5), p_4 = (4,9)$, $p_5 = (52,375), p_6 = (5234, 37866), p_7 = (8,-23), p_8 = (43,282)$ and $p_9 = (\frac{1}{4},-\frac{33}{8})$. Thus, we can find Prym-Brill-Noether general curves contained in very specific blow-ups of $\mathbb{P}^2$. 
\end{rmk}

\textbf{Acknowledgments:} I would like to thank Gavril Farkas, Martin M\"oller, Martin Ulirsch, Nicola Tarasca and Pedro Souza for discussions related to the topics appearing in this paper. The author acknowledges support by Deutsche Forschungsgemeinschaft (DFG, German Research Foundation) through the Collaborative Research Centre TRR 326 Geometry and Arithmetic of Uniformized Structures, project number 444845124.
\section{Introduction to Prym-Brill-Noether Theory}
Starting with a double cover $f\colon\widetilde{C}\rightarrow C$ corresponding to an element $[C, p_1+p_2+\cdots+p_{2k},\eta]$ in $\cR_{g,2k}$, the goal of Prym-Brill-Noether Theory is to study line bundles on $\widetilde{C}$, with sufficiently many sections, that take the map $f$ into account. To this setting we associate the norm map  
\[ \mathrm{Nm}_f\colon \mathrm{Pic}(\widetilde{C})\rightarrow \mathrm{Pic}(C)\]
sending a line bundle $L\in\mathrm{Pic}(\widetilde{C})$ to $\det(f_*L)\otimes \eta$.  Equivalently, it maps a line bundle $\OO_{\widetilde{C}}(D)$ to $\OO_C(f_*D)$ for any divisor $D$ on $\widetilde{C}$.

We aim to study line bundles on $\widetilde{C}$ of norm either $\omega_C$ or $\omega_C\otimes \eta$. From the perspective of Brill-Noether Theory, we are interested in those line bundles possessing sufficiently many sections. We will study the geometry of the Prym-Brill-Noether loci: 
\[ V^r(f) = \overline{\left\{L \in \textrm{Pic}^{2g-2}(\widetilde{C}) \ | \ \textrm{Nm}_f(L)= \omega_C, \  \ h^0(\widetilde{C}, L) = r+1 \ \right\}}.\]
and  
\[ V^r_\eta(f) \coloneqq \left\{L\in \mathrm{Pic}^{2g-2+k}(\widetilde{C})\ | \ \textrm{Nm}_f(L) = \omega_C\otimes \eta, \ \textrm{and} \ h^0(\widetilde{C}, L) \geq r+1 \right\}.\]

While little is known about the geometry of twisted Prym-Brill-Noether loci $V^r_\eta(f)$, significant research  has focused on understanding the geometry of $V^r(f)$ when $f\colon\widetilde{C}\rightarrow C$ corresponds to an element in $\cR_g$. This locus is contained in one of the components 
\[ P^+ = \left\{L\in \mathrm{Pic}^{2g-2}(\widetilde{C}) \ | \ \mathrm{Nm}(L) = \omega_C \ \mathrm{and} \ h^0(\widetilde{C}, L) \equiv 0\ (\mathrm{mod} \ 2) \right\} \]
or 
\[ P^- = \left\{L\in \mathrm{Pic}^{2g-2}(\widetilde{C}) \ | \ \mathrm{Nm}(L) = \omega_C \ \mathrm{and} \ h^0(\widetilde{C}, L) \equiv 1\ (\mathrm{mod} \ 2) \right\}, \]
where both are translations of a $(g-1)$-dimensional principally polarized Abelian variety. 
Furthermore, when $f\colon\widetilde{C} \rightarrow C$ is generic in $\cR_g$, the Prym-Brill-Noether locus $V^r(f)$ has dimension $g-1-\frac{r(r+1)}{2}$, see \cite{Welters} and \cite{Bertram}; is smooth away from the locus $V^{r+2}(f)$, see \cite{Welters}; and when $g - 1-\frac{r(r+1)}{2} \geq 1$, it is irreducible, see \cite{DebarreLefschetz}. 

\subsection{Prym-Brill-Noether loci as Lagrangian degenerations}
\label{sub: Lagrangian}

The goal of this subsection is to realize the Prym-Brill-Noether loci as Lagrangian degeneracy loci. In the notation above, we consider the diagram  
	\[
\begin{tikzcd}
P^\pm\times \widetilde{C}  \arrow{r}{id\times f} \arrow[swap]{d}{} & P^\pm\times C \arrow{d}{q}  \arrow{r}{\gamma} & P^\pm \\
	\widetilde{C} \arrow{r}{f}&  C
\end{tikzcd}
\]
Let $\mathcal{L} \rightarrow P^\pm\times \widetilde{C}$ be a Poincar\'e bundle normalized so that $\mathrm{Nm}(\mathcal{L}) \cong q^*(\omega_C)$ and let $\mathcal{E} \coloneqq (id\times f)_*\mathcal{L}$. For a reduced effective divisor $D$ on $C$ of sufficiently large degree $N$, we define the rank $4N$ vector bundle
\[ \mathcal{V} \coloneqq \gamma_*\text{\large(}\mathcal{E}(D)/\mathcal{E}(-D)\text{\large)},    \]
together with two subbundles of rank $2N$
\[ \mathcal{U} \coloneqq \gamma_*\text{\large(}\mathcal{E}/\mathcal{E}(-D)\text{\large)} \ \textrm{and} \ \mathcal{W} \coloneqq \gamma_*\text{\large(}\mathcal{E}(D)\text{\large)}.\]

A key insight from Mumford, see \cite{MumfordPrym}, is that the norm condition provides a non-degenerate quadratic form 
\[\mathcal{E} \rightarrow q^*\omega_C.\]
This form induces a non-degenerate quadratic form on the vector bundle $\mathcal{V}$, for which $\mathcal{U}$ and $\mathcal{W}$ are maximal isotropic subbundles, see \cite{DeConciniPragacz}. 

We immediately notice the equality of spaces 
\[H^0(\widetilde{C}, L) = H^0(C, E) = H^0(C, E(D)) \cap H^0(C,E/E(-D)),\] 
viewed as an intersection of subspaces of $H^0(C, E(D)/E(-D))$. In particular, the locus $V^r(f)$ is equal to the locus 
\[ \left\{L \in P^\pm \ | \ \dim(\mathcal{U}\cap\mathcal{W})_{L} \geq r+1\right\}. \]

This description of the Prym-Brill-Noether locus $V^r(f)$ as a Lagrangian degeneration locus implies that its dimension is at least the expected dimension $g-1-\frac{r(r+1)}{2}$. When the dimension is $g-1-\frac{r(r+1)}{2}$ as expected, this description of $V^r(f)$ allows us to compute its class in the numerical equivalence ring $N^*(P^\pm, \mathbb{C})$ or the singular cohomology  $H^*(P^\pm, \mathbb{C})$. We obtain the formula
\[ [V^r(f)] = 2^{\frac{r(r+1)}{2}} \cdot \prod_{i=1}^{r} \frac{i!}{(2i)!}\cdot \xi^{\frac{r(r+1)}{2}} \]
where $\xi$ is the class of the theta divisor of $P^\pm$, see \cite{DeConciniPragacz}. This result was extended by Tarasca to Prym-Brill-Noether loci with prescribed vanishing at a point, see \cite{tarasca-pointed-prym}. Specifically, for a vanishing sequence $\textbf{a} = (0\leq a_0< a_1 <\cdots < a_r\leq 2g-2)$ and a point $p$ in $ \widetilde{C}$, we can represent the locus 
\[
\mathrm{V}^{\textbf{a}}(f, p) := \left\{ L\in \mathrm{Pic}^{2g-2}(\widetilde{C}) \, \left\vert \, 
\begin{array}{l}
	\mathrm{Nm}(L)= \omega_C,  \\[0.2cm]
	h^0(\widetilde{C},L)\equiv r+1 \mbox{ mod } 2, \\[0.2cm]
	h^0(\widetilde{C}, L(-a_i\,p))\geq r+1-i, \,\, \forall \, 0\leq i \leq r 
\end{array}
\right.
\right\}
\] 
as a Lagrangian degeneracy locus. By defining 
\[ \mathcal{E}_i \coloneqq (id\times f)_*(\mathcal{L}-a_ip) \ \textrm{and} \ \mathcal{W}_i \coloneqq \gamma_*\text{\large(}\mathcal{E}_i(D)\text{\large)},\]
we can describe $\mathrm{V}^{\textbf{a}}(f, p)$ as the locus 
\[ \left\{L \in P^\pm \ | \ \dim(\mathcal{U}\cap\mathcal{W}_i)_{L} \geq r+1-i \ \forall\  0\leq i \leq r\right\}. \]
Using this description, Tarasca computed the class of $V^\textbf{a}(f,p)$ in $N^*(P^\pm, \mathbb{C})$ and $H^*(P^\pm, \mathbb{C})$.

This Lagrangian degeneration has not been investigated for twisted Prym-Brill-Noether loci. In Section \ref{sec: twist}, we provide the twisted counterparts to the results in \cite{DeConciniPragacz} and \cite{tarasca-pointed-prym} regarding dimension and class computation.

\section{Brill-Noether and Gieseker-Petri conditions}
\label{sec: BN-GP}
We consider the double cover $f\colon \widetilde{C}\rightarrow C$ associated to a generic element in $\mathcal{R}_{g,2k}$ for $1\leq k \leq 3$ and investigate the Brill-Noether and Gieseker-Petri properties of $\widetilde{C}$. Our goal is to extend the results of \cite{Lelli-Chiesa-lowgenus} and \cite[Section 4]{BudKodPrym}. We will study how the curve $\widetilde{C}$ together with two points on it behaves with respect to the coupled Gieseker-Petri condition. This condition can be used to study the geometry of Brill-Noether loci for twice marked curves, see \cite{PfluegerVersatility}. In our context, it can be used to study twisted Prym-Brill-Noether loci of marked Prym curves, see Section \ref{sec: twist}.

\begin{defi} Let $[C,p,q] \in \cM_{g,2}$ and let $L$ be a line bundle on $C$. We define the space of coupled tensors as 
	\[ T^L_{p,q} \coloneqq \sum_{(a,b)\in \mathbb{Z}\times \mathbb{Z}} H^0(C, L(-ap-bq))\otimes H^0(C, \omega_C\otimes L^\vee(ap+bq)). \]
In this definition, we regard each term as a subspace of $H^0(C^*,L) \otimes H^0(C^*, \omega_C\otimes L^\vee)$ where $C^* = C\setminus \left\{p,q\right\}$. 
\end{defi}

\begin{defi}
	We define the fully coupled Gieseker-Petri map as
	\[ \mu^L_{p,q} \colon T^L_{p,q} \rightarrow H^0(C,\omega_C). \]
	We say that $[C,p,q]$ satisfies the coupled Gieseker-Petri condition if the map $\mu^L_{p,q}$ is injective for every $L\in \textrm{Pic}(C)$. 
\end{defi}

The coupled Gieseker-Petri condition is important as it guarantees that all pointed Brill-Noether loci are of the expected dimension. For the source curve of a double cover, we have that: 


\begin{trm} \label{trm: coupled k = 1}
	Let $f\colon\widetilde{C}\rightarrow C$ be a double cover corresponding to a very general element of $\mathcal{R}_{g,2}$ and let $y_1,y_2$ be the points of $\widetilde{C}$ in the preimage $f^{-1}(x)$ for $x$ a very general point of $C$. Then the curve $[\widetilde{C}, y_1,y_2]$ satisfies the coupled Gieseker-Petri condition.  
\end{trm}

\begin{proof} It suffices to find a unique pointed double cover for which the coupled Gieseker-Petri condition is satisfied. 
	
    We consider $[X\cup_p R, x_1+x_2, \OO_X, \eta_R]$ an element in the boundary component of $\overline{\mathcal{R}}_{g,2}$ satisfying that $R$ is a rational curve, $\eta_{|X} = \OO_X$ and $\eta_{|R}^{\otimes 2} \cong \OO_R(x_1 + x_2)$. 
	
	After stabilization, the source of the double cover is the curve $[X_1\cup_{p_1\sim p_2}X_2]$ consisting of two copies $[X_1,p_1]$ and $[X_2,p_2]$ of $[X,p]$ glued together. 
	
	For a very general choice of $[X,p]$ and a very general point $x \in X$, the curves $[X_1, y_1, p_1]$ and $[X_2, y_2, p_2]$ satisfy the coupled Gieseker-Petri condition, see \cite[Remark 1.8]{PfluegerVersatility}. Furthermore, a chain of curves satisfying the coupled Gieseker-Petri condition continues to satisfy this condition, see \cite[Theorem 2.2]{PfluegerVersatility}. This concludes the proof.
\end{proof}

\begin{rmk} \label{rmk: BN-GP k = 2} Using the same degeneration as in Theorem \ref{trm: coupled k = 1} above, we can prove that $\widetilde{C}$ is Brill-Noether and Gieseker-Petri general, see \cite[Theorem 4.1]{BudKodPrym}.
\end{rmk}

Using the boundary description of $\rr_{g,2k}$ appearing in \cite{BudKodPrym}, we obtain the following:

\begin{trm} \label{BN-GP general}
	Let $f\colon\widetilde{C} \rightarrow C$ be a double cover corresponding to a general element in $\cR_{g,4}$. Then the curve $\widetilde{C}$ is both Brill-Noether general and Gieseker-Petri general. Moreover, if $\widetilde{x}$ is one of the four ramification points, the marked curve $[\widetilde{C},\widetilde{x}] \in \mathcal{M}_{2g+1,1}$ is pointed Brill-Noether general.  
\end{trm}

\begin{proof}
	Since the conditions of being pointed Brill-Noether general or Gieseker-Petri general are open, it suffices to exhibit a unique double cover satisfying these conditions. 
	
	We consider a generic element $[X,p] \in \mathcal{M}_{g,1}$, and take a pointed curve 
	\[ [X\cup_p R_1\cup_q R_2, x_1,x_2, x_3, x_4] \] 
	where $R_1, R_2$ are rational components satisfying $x_1 \in R_1$ and $[R_2,q,x_2,x_3,x_4] \in \mathcal{M}_{0,4}$ is generic.
	Let $\eta$ be the line bundle on this curve whose restrictions to the components satisfy: 
	\[ \eta_X = \OO_X, \ \ \eta_{R_1}^{\otimes2} \cong \OO_{R_1}(x_1 + q), \ \textrm{and} \ \eta_{R_2}^{\otimes2} \cong \OO_{R_2}(q+x_2+x_3+x_4).  \] 
	
	Out of this data, we obtain 
	\begin{enumerate}
		\item A double cover $\widetilde{R}_1 \rightarrow R_1$ associated to $[R_1, x_1, q, \eta_{R_1}]$, and we denote by $\widetilde{q}$ the preimage of $q$ and by $p_1, p_2$ the two preimages of $p$; 
		\item A double cover $\widetilde{E} \rightarrow R_2$ associated to $[R_2, q, x_2, x_3, x_4, \eta_{R_2}]$ and we denote with $\widetilde{q}$ the preimage of $q$ in $\widetilde{E}$. 
	\end{enumerate}
	
	We can now describe the double cover associated to $ [X\cup_p R_1\cup_q R_2, x_1,x_2, x_3, x_4, \eta]$. Let $ [X_1, p_1]$ and $[X_2, p_2]$ be two copies of the curve $[X,p]$. The source of the double cover is obtained by glueing to the rational curve $[\widetilde{R_1}, p_1, p_2, \widetilde{q}]$ the curves $[X_1,p_1]$, $[X_2, p_2]$ and $[\widetilde{E},\widetilde{q}]$ so that the markings denoted by the same symbol are identified. 
	
	Because the curves $[X_1, p_1]$, $[X_2, p_2]$, $[\widetilde{E},\widetilde{q}]$ and $[\widetilde{R_1}, \widetilde{q}, \widetilde{x}_1,p_1, p_2]$ are all pointed Brill-Noether general, it follows that the curve obtained by glueing them together is pointed Brill-Noether general, see \cite[Theorem 1.1]{EisenbudHarrisg>23}. 
	
	Moreover, by degenerating the curve $[X,p]$ to a chain of rational components, each glued to a unique elliptic component, we obtain from \cite[Theorem $A'$]{EisenHarrisGieseker} that this curve is Gieseker-Petri general. 
\end{proof}

Next, let $f\colon \widetilde{C}\rightarrow C$ be a generic element in $\cR_{g,4}$. We aim to determine whether the curve $\widetilde{C}$ together with two very general points in the same fiber of $f$ satisfies the coupled Gieseker-Petri condition.

\begin{trm} \label{trm: coupled k = 2}
	Let $f\colon\widetilde{C}\rightarrow C$ be a double cover corresponding to a very general element of $\mathcal{R}_{g,4}$ and let $y_1,y_2$ be the points of $\widetilde{C}$ in the preimage $f^{-1}(x)$ for $x$ a very general point of $C$. Then the pointed curve $[\widetilde{C}, y_1,y_2]$ satisfies the coupled Gieseker-Petri condition.  
\end{trm}
\begin{proof}
   As before, it suffices to find a unique pointed curve $[\widetilde{C},y_1, y_2]$ satisfying the coupled Gieseker-Petri condition. 
   
   Let $[X,p] \in \cM_{g,1}$ be very general and consider the pointed curve 
   \[ [X\cup_p R, x_1, x_2,x_3, x_4]\]
   where $R$ is a rational component containing the points $x_i$. Let $\eta$ be a line bundle on this curve whose restrictions to the components satisfy 
   \[ \eta_X  = \OO_X \ \textrm{and} \ \eta_R^{\otimes 2} \cong \OO_R(x_1+x_2+x_3+x_4)\] 
   
   Let $\widetilde{E} \rightarrow R$ be the double cover associated to $[R, x_1, x_2, x_3, x_4, \eta_{R}]$ and we denote by $p_1, p_2$ the points in the preimage of $p$ in $\widetilde{E}$. 
   
   The source of the double cover associated to $[X\cup_p R, x_1,x_2,x_3, x_4, \eta]$ consists of two copies $[X_1,p_1]$ and $[X_2,p_2]$ glued to the elliptic curve $[\widetilde{E},p_1,p_2]$. We consider $q$ to be a very general point of $X$ and denote $q_1, q_2$ the corresponding points in $X_1$ and $X_2$ respectively. 
   
   Since $p_1-p_2$ is not torsion, we know that $[\widetilde{E},p_1,p_2]$ satisfies the coupled Gieseker-Petri condition, see \cite[Example 1.6]{PfluegerVersatility}. Because $[X_1,p_1,q_1]$ and $[X_2,p_2,q_2]$ are very general, these curves also satisfy the coupled Gieseker-Petri condition, see \cite[Remark 1.8]{PfluegerVersatility}. A chain of curves satisfying the coupled Gieseker-Petri condition also satisfies the coupled Gieseker-Petri condition, see \cite[Theorem 2.2]{PfluegerVersatility}, hence the conclusion. 
\end{proof}

When $f\colon\widetilde{C} \rightarrow C$ is the double cover corresponding to a generic element of $\cR_{g,6}$, we recover the result from \cite{Lelli-Chiesa-lowgenus}:
\begin{trm} \label{ram=6}
	Let $f\colon\widetilde{C} \rightarrow C$ be the double cover corresponding to a generic element of $\cR_{g,6}$. Then the curve $\widetilde{C}$ is Brill-Noether general. 
\end{trm}

\begin{proof}
	We consider $\widetilde{X} \rightarrow X$ be a generic double cover of a curve of genus $g\geq 2$, ramified at $4$ points, and $\widetilde{E} \rightarrow \mathbb{P}^1$ a generic double cover of a rational curve, ramified at $4$ points. 
	
	If we glue a ramification point in $\widetilde{X}$ to a ramification point in $\widetilde{E}$, we obtain a double cover \[[\widetilde{X}\cup_{\widetilde{q}}\widetilde{E} \rightarrow X\cup_q \mathbb{P}^1]\] of a genus $g$ curve, ramified at $6$ points. 
	
	From Theorem \ref{BN-GP general}, we know that $[\widetilde{X},\widetilde{q}]$ is pointed Brill-Noether general. Because $[\widetilde{E}, \widetilde{q}]$ is also pointed Brill-Noether general, the conclusion follows from \cite[Theorem 1.1]{EisenbudHarrisg>23}.
\end{proof}
\section{Prym-Brill-Noether Theory for ramified double covers}
\label{sec: PBN ramified}
The goal of this section is to study the geometry of the Prym-Brill-Noether loci $V^r(f)$ when $f\colon \widetilde{C} \rightarrow C$ corresponds to a generic element in $\cR_{g,2}$. We prove that $V^r(f)$ has the expected dimension and moreover, it is smooth away from the locus $V^{r+1}(f) \cap V^r(f)$. To prove this result, we consider the clutching map 
\[ i\colon \cR_{g,2} \rightarrow \Delta^{\textrm{ram}}_0\subseteq \rr_{g+1}, \ [f\colon \widetilde{C} \rightarrow C] \mapsto [f'\colon \widetilde{C}_{/\widetilde{x}\sim\widetilde{y}} \rightarrow C_{/x\sim y} ] \]
obtained by glueing together the two branch points $x, y \in C$ and the two ramification points $\widetilde{x}, \widetilde{y} \in \widetilde{C}$. Using this map, we can relate the geometry of $V^r(f)$ to the geometry of Prym-Brill-Noether loci in $\rr_{g+1}$. 

We remark that, via Serre duality, we can identify the locus $ V^r(f)$ with the locus
\[ V^{r+1}(f, x+y) \coloneqq \overline{\left\{ L \in \textrm{Pic}^{2g}(\widetilde{C}) \ | \ \textrm{Nm}_f(L) = \omega_C(x+y),\ h^0(\widetilde{C}, L) = r+2 \right\}}, \]
where $x, y \in C$ are the branched points of $f$. This locus appears as the degeneration of a Prym-Brill-Noether locus via the map $i\colon \cR_{g,2} \rightarrow \Delta_0^{\textrm{ram}} \subseteq  \rr_{g+1}$. As such, many geometric properties of $V^{r+1}(f, x+y)$ can be derived from the geometry of Prym-Brill-Noether loci on $\rr_{g+1}$. We will use the methods of \cite{Welters} to study $ V^{r+1}(f, x+y)$.

Let $[f\colon\widetilde{C}\rightarrow C]$ be generic in $\mathcal{R}_{g,2}$ and let $x,y\in C$ and $\widetilde{x}, \widetilde{y}\in \widetilde{C}$ be the branch points and ramification points of $f$, respectively. The map $f$ determines an involution $\iota\colon\widetilde{C}\rightarrow \widetilde{C}$.  

We consider the locus 
\[\left\{ L \in \textrm{Pic}^{2g}(\widetilde{C}) \ | \ \textrm{Nm}_f(L) = \omega_C(x+y)\right\}\]
which is the translation of a $g$-dimensional Abelian variety. A line bundle $L$ in this locus satisfies 
\[L \otimes \iota^* L \cong \omega_{\widetilde{C}}(\widetilde{x} + \widetilde{y})\] 

We can define a ramified Prym-Gieseker-Petri map as in \cite[Remark (1.12)]{Welters}. For this, we consider the composition: 
\[ H^0(\widetilde{C}, L)\otimes H^0(\widetilde{C}, L) \xrightarrow{id\otimes\iota^*} H^0(\widetilde{C}, L)\otimes H^0(\widetilde{C}, \iota^*L) = H^0(\widetilde{C}, L)\otimes H^0\text{\large(}\widetilde{C}, \omega_{\widetilde{C}}\otimes L^\vee(\widetilde{x}+\widetilde{y})\text{\large)} \rightarrow H^0\text{\large(}\widetilde{C}, \omega_{\widetilde{C}}(\widetilde{x}+\widetilde{y})\text{\large)}  \] 
Looking at the invariant and anti-invariant sections of  $ H^0\text{\large(}\widetilde{C}, \omega_{\widetilde{C}}(\widetilde{x}+\widetilde{y})\text{\large)}$, we can write it as 
\[   H^0\text{\large(}\widetilde{C}, \omega_{\widetilde{C}}(\widetilde{x}+\widetilde{y})\text{\large)} = H^0\text{\large(}C,\omega_C(x+y)\text{\large)}\oplus H^0(C,\omega_C\otimes \eta) \]

At the level of anti-invariant sections, the map 
\[H^0(\widetilde{C}, L)\otimes H^0(\widetilde{C}, L) \rightarrow   H^0\text{\large(}\widetilde{C}, \omega_{\widetilde{C}}(\widetilde{x}+\widetilde{y})\text{\large)} \]
restricts to the ramified Prym-Gieseker-Petri map: 
\[\wedge^2H^0(\widetilde{C}, L) \rightarrow H^0(C,\omega_C\otimes \eta)\]
Similarly to \cite{Welters}, we can use the ramified Prym-Gieseker-Petri map to understand the tangent space of $V^r(f)$ at a point $[L]$. We will show that, for a generic $f\colon\widetilde{C} \rightarrow C$ in $\cR_{g,2}$, the ramified Prym-Gieseker-Petri map is always injective:  
\begin{prop} \label{prop: prym-petri 2-ram}
	Let  $f\colon \widetilde{C}\rightarrow C$ be a generic element of $\mathcal{R}_{g,2}$ and denote by $x, y$ its branch points and by $\eta\in \textrm{Pic}^1(C)$ its associated torsion line bundle. Then, for any element $L \in \textrm{Pic}^{2g}(\widetilde{C})$ satisfying $\textrm{Nm}_f(L) = \omega_C(x+y)$, the ramified Prym-Gieseker-Petri map  
	\[\wedge^2H^0(\widetilde{C}, L) \rightarrow H^0(C,\omega_C\otimes \eta)\]
	is injective. 
\end{prop}
\begin{proof} Similarly to \cite[Section 2]{Welters}, it is sufficient to exhibit a unique double cover for which the ramified Prym-Gieseker-Petri map is injective:
	
We consider $C \coloneqq E_1\cup_{p_1} E_2 \cup \cdots E_g \cup_{p_g} R$ a chain of $g$ elliptic curves and a rational curve $R$. Furthermore we assume that $p_i-p_{i-1}$ is not torsion on $E_i$ for any $2 \leq i \leq g$ and assume the points $x, y$ are on $R$. For the torsion line bundle $\eta$ defining a branched double cover on the curve, we take it to be trivial on the $E_i$'s and to satisfy $\eta_{|R}^{\otimes2} \cong \OO_R(x+y)$.
	
	The associated double cover will be of the form 
	\[\widetilde{C} = E^1_1\cup_{p^1_1} E^1_2 \cup \cdots E^1_g \cup_{p^1_g} \widetilde{R} \cup_{p^2_g} E^2_g\cup_{p^2_{g-1}} E^2_{g-1} \cup \cdots \cup E^2_1 \rightarrow C \coloneqq E_1\cup_{p_1} E_2 \cup \cdots E_g \cup_{p_g} R \]
	where $E^1_1\cup_{p^1_1} E^1_2 \cup \cdots E^1_g$ and $E^2_g\cup_{p^2_{g-1}} E^2_{g-1} \cup \cdots \cup E^2_1$ are two copies of $E_1\cup_{p_1} E_2 \cup \cdots E_g$ and $\widetilde{R} \rightarrow R$ is the map determined by $\eta_{|R}$, with $p^1_g$ and $p^2_g$ the points in the preimage of $p_g$. 
	
	The same method as in \cite{Welters} can be applied to this double cover to obtain the conclusion. One can view this result as the one in \cite{Welters} with the elliptic bridge degenerating to $\widetilde{R}/\widetilde{x}\sim \widetilde{y}$.  
\end{proof}

Using the bijective morphism 
\[ V^r(f) \rightarrow V^{r+1}(f, x+y), \ \  L \mapsto \iota^*L\]
we obtain the desired result:
\begin{trm} \label{trm: dim-ramified}
	Let $r\geq 0$ and let $[f\colon \widetilde{C} \rightarrow C]$ be a generic element of $\mathcal{R}_{g,2}$. Then the Prym-Brill-Noether locus $V^r(f)$ has dimension $g-\frac{(r+1)(r+2)}{2}$ and is smooth away from the locus of line bundles having strictly more than $r+1$ independent global sections.   
\end{trm}
\begin{proof} We divide the proof into three parts:
	\begin{itemize}
		\item Let $[L] \in V^r(f)$ with $h^0(\widetilde{C}, L) = r+1$. We start by showing $V^r(f)$ is smooth of dimension $g-\frac{(r+1)(r+2)}{2}$ at $[L]$. 
		\item We then show $V^r(f)$ is non-empty whenever $g \geq \frac{(r+1)(r+2)}{2}$. 
		\item Lastly, we show that $V^r(f)$ is empty whenever  $g < \frac{(r+1)(r+2)}{2}$.
	\end{itemize}
    	When $\textrm{Nm}_f(L) = \omega_C$, we have the Prym-Petri map 
    \[ v\colon H^0(\widetilde{C},L)\otimes H^0\text{\large(}\widetilde{C}, L(\widetilde{x}+\widetilde{y})\text{\large)}\rightarrow H^0(\widetilde{C}, \omega_{\widetilde{C}})^-\]
    As in \cite{Welters} and \cite{Kanev}, the tangent space $T_L(V^r(f))$ is identified with $\textrm{Im}(v)^\perp$. Hence, it suffices to show that $\textrm{Ker}(v) = S^2H^0(\widetilde{C}, L)$ to conclude $V^r(f) \setminus V^{r+1}(f)$ is smooth of dimension $g-\frac{(r+1)(r+2)}{2}$. This map fits into a commutative diagram 
    
    \[
    \begin{tikzcd}
    	H^0(\widetilde{C},L)\otimes H^0\text{\large(}\widetilde{C}, L(\widetilde{x}+\widetilde{y})\text{\large)} \arrow{r}{v}  \arrow[hookrightarrow]{d}{i} & H^0(\widetilde{C},\omega_{\widetilde{C}})^- \arrow[hookrightarrow]{d}{i} \\
    	H^0\text{\large(}\widetilde{C},L(\widetilde{x}+\widetilde{y})\text{\large)}\otimes H^0\text{\large(}\widetilde{C}, L(\widetilde{x}+\widetilde{y})\text{\large)}\arrow{r}{}& H^0\text{\large(}\widetilde{C},\omega_{\widetilde{C}}(\widetilde{x}+\widetilde{y})\text{\large)}^-
    \end{tikzcd}
    \] 
    and the kernel of the map downstairs is $S^2H^0\text{\large(}\widetilde{C},L(\widetilde{x}+\widetilde{y})\text{\large)}$, see Proposition \ref{prop: prym-petri 2-ram}. In particular 
    \[\textrm{Ker}(v) = S^2H^0\text{\large(}\widetilde{C}, L(\widetilde{x}+\widetilde{y})\text{\large)}\cap \text{\Large(}H^0(\widetilde{C},L)\otimes H^0\text{\large(}\widetilde{C}, L(\widetilde{x}+\widetilde{y})\text{\large)}\text{\Large)} = S^2H^0(\widetilde{C}, L)\]
    In conclusion, the tangent space of $V^r(f) \setminus V^{r+1}(f)$ has dimension $g-\frac{(r+1)(r+2)}{2}$ at any point. 
        
	Next we show that $V^r(f)$ is non-empty when $g-\frac{(r+1)(r+2)}{2} \geq 0$. For this we consider the map 
	\[ \pi\colon \cR_{g,2} \rightarrow \Delta_0^{\textrm{ram}}\subseteq \cR_{g+1} \]
	sending the double cover $f\colon \widetilde{C} \rightarrow C$ to $f'\colon \widetilde{C}_{/\widetilde{x}\sim\widetilde{y}} \rightarrow C_{/x\sim y} $. 
	
	For a generic element of $\cR_{g+1}$, an element of a Prym-Brill-Noether locus $V^{r+1}$ is simply a linear series $g^{r+1}_{2g}$ satisfying the norm condition. As such, when we degenerate the double cover to some $f'\colon \widetilde{C}_{/\widetilde{x}\sim\widetilde{y}} \rightarrow C_{/x\sim y} $ generic in the boundary divisor $\Delta_0^{\textrm{ram}}$, the limit of elements in $V^{r+1}$ must be limit linear series $g^{r+1}_{2g}$ satisfying the norm condition. We are interested in understanding the locus of Prym limit linear series, i.e. limit linear series $g^{r+1}_{2g}$ satisfying the norm condition. For a generic $f'\colon \widetilde{C}_{/\widetilde{x}\sim\widetilde{y}} \rightarrow C_{/x\sim y} $ in $\Delta_0^{\textrm{ram}}$, this locus can be identified with 
	\[\left\{ L \in \textrm{Pic}^{2g}(\widetilde{C}) \ | \ \textrm{Nm}_f(L) = \omega_C(x+y), h^0(\widetilde{C}, L) \geq r+2 \ \textrm{and} \  h^0(\widetilde{C}, L(-\widetilde{x}-\widetilde{y})) \geq r+1\right\} \]
	via normalizing the map $f'$. 
	
	Notice that the last condition is superfluous: Because $\textrm{Nm}_f(L-\widetilde{x}) = \omega_C(y)$ and all global sections of $\omega_C(y)$ vanish at $y$, it follows that all global sections of $L(-\widetilde{x})$ vanish at $\widetilde{y}$. In particular, the locus of Prym limit $g^{r+1}_{2g}$ on $f'$ is identified with 
	\[\left\{ L \in \textrm{Pic}^{2g}(\widetilde{C}) \ | \ \textrm{Nm}_f(L) = \omega_C(x+y), \ h^0(\widetilde{C}, L) \geq r+2 \right\}. \]
    Assume that all Prym limit $g^{r+1}_{2g}$ appearing as degenerations of Prym-Brill-Noether loci $V^r$ have at least $r+3$ independent global sections. This would imply 
	\[ \dim V^{r+s}(f,x+y) \geq g-\frac{(r+1)(r+2)}{2} \]
	for some $s\geq 2$. But the description of the tangent space above immediately implies 
	 \[\dim V^{r+s}(f,x+y)  \leq g-\frac{(r+s)(r+s+1)}{2}. \]
	Consequently, the locus $\dim V^{r+1}(f,x+y)$ is non-empty when $g-\frac{(r+1)(r+2)}{2}\geq 0$ as it must contain the normalization of the Prym limit $g^{r+1}_{2g}$ on $\Delta_0^{\textrm{ram}}\subseteq \rr_{g+1}$. 
	
    Using again the Serre duality 
	\[V^{r+1}(f,x+y) \cong V^r(f)\]
	we conclude that $V^r(f)$ is non-empty and has dimension at least $g-\frac{(r+1)(r+2)}{2}$.

	Lastly we prove that $V^r(f)$ is empty when $g-\frac{(r+1)(r+2)}{2} < 0$. All elements of $V^r(f)$ are $g^r_{2g-2}$ on the curve $\widetilde{C}$ of genus $2g$. But $\rho(2g, r, 2g-2) = 2g - (r+1)(r+2) < 0$ and the curve $\widetilde{C}$ is Brill-Noether general, see Remark \ref{rmk: BN-GP k = 2} and \cite[Theorem 4.1]{BudKodPrym}. This implies $V^r(f) = \emptyset$.
\end{proof}

\section{Twisted Prym-Brill-Noether loci}
\label{sec: twist}

In this section, we study the geometry of the twisted Prym-Brill-Noether loci
 \[ V^r_\eta(f) \coloneqq \left\{L\in \mathrm{Pic}^{2g-2+k}(\widetilde{C})\ | \ \textrm{Nm}_f(L) = \omega_C\otimes \eta \ \textrm{and} \ h^0(\widetilde{C}, L) \geq r+1 \right\} \]
where $f\colon\widetilde{C}\rightarrow C$ is generic in $\cR_{g,2k}$ for $0\leq k \leq 2$ and $\eta$ is its associated torsion line bundle. Analogously to Section \ref{sec: PBN ramified}, we will compute the dimension of $V^r_\eta(f)$ by understanding the tangent space at its points. This tangent space is described using twisted Prym-Gieseker-Petri maps, similar to Section \ref{sec: PBN ramified} and \cite{Welters}. The injectivity of the twisted Prym-Gieseker-Petri map is a consequence of the fact that $\widetilde{C}$ is Gieseker-Petri general for $f\colon \widetilde{C} \rightarrow C$ generic in $\cR_{g,2}$ or $\cR_{g,4}$. 

More generally, for a vanishing sequence $\textbf{a} = (0\leq a_0 < a_1 < \cdots< a_r\leq 2g-2+k)$, we consider the twisted Prym-Brill-Noether locus with prescribed vanishing $\textbf{a}$ at a point $p\in \widetilde{C}$: 

	\[
\mathrm{V}_\eta^{\textbf{a}}(f, p) := \left\{ L\in \mathrm{Pic}^{2g-2+k}(\widetilde{C}) \, \left\vert \, 
\begin{array}{l}
	\mathrm{Nm}(L)= \omega_C\otimes \eta,  \\[0.2cm]
	h^0\text{\large(}\widetilde{C}, L(-a_i\,p)\text{\large)}\geq r+1-i, \,\, \forall \, 0\leq i \leq r 
\end{array}
\right.
\right\}.
\]

Using the coupled Gieseker-Petri generality, see Theorem \ref{trm: coupled k = 1} and Theorem \ref{trm: coupled k = 2}, we can prove that this locus is of the expected dimension 
\[ \dim V^\textbf{a}_\eta(f,p) = g+ k - r-2- \sum_{i=0}^r a_i \]
when $f$ and $p\in \widetilde{C}$ are generic. We will denote 
\[ P \coloneqq \left\{L\in \mathrm{Pic}^{2g-2+k}(\widetilde{C}) \ | \ \mathrm{Nm}(L) = \omega_C\otimes \eta \right\}.\] 
Furthermore, we will compute the class of $\mathrm{V}_\eta^{\textbf{a}}(f, p)$ in the numerical equivalence ring $N^*(P, \mathbb{C})$ and the singular cohomology  $H^*(P, \mathbb{C})$.

Similar to the description in Subsection \ref{sub: Lagrangian}, we can realize the twisted Prym-Brill-Noether loci as Lagrangian degeneration loci. In the notation of Subsection \ref{sub: Lagrangian} modified to the twisted case, we have the isomorphism $\textrm{Nm}(\mathcal{L}) \cong q^*(\omega_C\otimes \eta)$. Looking at the vector bundle  $\mathcal{E} \coloneqq (id\times f)_*\mathcal{L}$, the norm condition implies 
\[ \wedge^2 \mathcal{E} \cong q^*\omega_C, \]
and hence we obtain a symplectic form  $\mathcal{E} \rightarrow q^*\omega_C$. 

As in Subsection \ref{sub: Lagrangian}, this symplectic form extends to a symplectic form on the vector bundle $\mathcal{V}$, with $\mathcal{U}, \mathcal{W}$ maximal isotropic subbundles of $\mathcal{V}$. In this setting we see that $V_\eta^r(f)$ is the locus 
\[ \left\{L \in P \ | \ \dim(\mathcal{U}\cap\mathcal{W})_{L} \geq r+1\right\} \]
and $\mathrm{V}_\eta^{\textbf{a}}(f, p)$ is 
\[ \left\{L \in P \ | \ \dim(\mathcal{U}\cap\mathcal{W}_i)_{L} \geq r+1-i \ \forall\  0\leq i \leq r\right\}. \]

These Lagrangian degeneracy loci are studied in \cite{Anderson-Fulton} and we immediately obtain one of the inequalities for the dimension estimate: The locus $V^\textbf{a}_\eta(f,p)$ is either empty or satisfies
\[ \dim V^\textbf{a}_\eta(f,p) \geq g+ k - r-2- \sum_{i=0}^r a_i = \dim(P) - (r+1) -\sum_{i=0}^ra_i. \]
For the converse inequality when $k = 1$ or $2$, we will look at the dimension of the tangent space and obtain:

\begin{trm} \label{trm: twist-dim} Let $f\colon \widetilde{C}\rightarrow C$ a generic element in $\cR_{g,2}$ and $\eta$ the torsion line bundle defining the double cover. Then the twisted Prym-Brill-Noether locus $ V^r_\eta(f) $ has dimension $g-\frac{(r+1)(r+2)}{2}$ and is smooth away from $V^{r+1}_\eta(f)$. 
\end{trm}

\begin{proof} Let $L \in V^r_\eta(f) \subseteq \textrm{Pic}^{2g-1}(\widetilde{C})$ satisfying $h^0(\widetilde{C},L) = r+1$. Our goal is to compute the dimension of the tangent space $T_L(V^r_\eta(f))$. 
	
	We know from \cite{Kanev} that $T_L(V^r_\eta(f))$ is identified with  $\textrm{Im}(v_1)^\perp$ where $v_1$ is the composition map 
	\[ H^0(\widetilde{C}, L) \otimes H^0(\widetilde{C}, \omega_{\widetilde{C}}\otimes L^\vee) \rightarrow H^0(\widetilde{C}, \omega_{\widetilde{C}}) \rightarrow H^0(\widetilde{C}, \omega_{\widetilde{C}})^-. \]

	Let $\iota\colon \widetilde{C} \rightarrow \widetilde{C}$ be the involution associated to the map $f$. The condition $\textrm{Nm}_f(L) = \omega_C\otimes \eta$ implies $L\otimes \iota^*L \cong \omega_{\widetilde{C}}$, and hence $\omega_{\widetilde{C}}\otimes L^\vee \cong \iota^*L$. 
	By composing our morphism with the isomorphism 
	\[ H^0(\widetilde{C}, L) \otimes H^0(\widetilde{C}, L) \xrightarrow{1\otimes \iota^*} H^0(\widetilde{C}, L) \otimes H^0(\widetilde{C},\omega_{\widetilde{C}}\otimes L^\vee) \]
	we conclude that $T_L(V^r_\eta(f))$ coincides with  $\textrm{Im}(v)^\perp$ for $v$ the map 
		\[ v\colon H^0(\widetilde{C}, L) \otimes H^0(\widetilde{C}, L) \rightarrow H^0(\widetilde{C}, \omega_{\widetilde{C}})^-  \] 
	By excluding the invariant sections, that are mapped to $0$ via $v$, the map above has the same image as the map 
	\[ S^2H^0(\widetilde{C}, L) \rightarrow H^0(\widetilde{C}, \omega_{\widetilde{C}})^-  \]
	We know from \cite[Theorem 4.1]{BudKodPrym} that for a generic $f\colon \widetilde{C} \rightarrow C$ in $\cR_{g,2}$, the curve $\widetilde{C}$ is Gieseker-Petri general. In particular, the map 
	\[ H^0(\widetilde{C}, L) \otimes H^0(\widetilde{C}, \omega_{\widetilde{C}}\otimes L^\vee) \rightarrow H^0(\widetilde{C}, \omega_{\widetilde{C}}) \] 
	is injective. On the other hand, the map 
	\[ S^2H^0(\widetilde{C}, L) \rightarrow H^0(\widetilde{C}, \omega_{\widetilde{C}})^+  \]
	is identically $0$.  
	These two results imply that 
	\[ S^2H^0(\widetilde{C}, L) \rightarrow H^0(\widetilde{C}, \omega_{\widetilde{C}})^-  \]
	is injective, and hence $T_L(V^r_\eta(f))$ has dimension $g-\frac{(r+1)(r+2)}{2}$. This implies the conclusion. 
\end{proof}

We are left to show that $V^r_\eta(f)$ is non-empty as soon as $g-\frac{(r+1)(r+2)}{2}\geq 0$. The methods of de Concini and Pragacz from \cite{DeConciniPragacz} apply to our situation. We already know that $V_\eta^r(f)$ is either empty or of expected dimension, hence we can use the formula in \cite[Type C: symplectic bundles, Corollary]{Anderson-Fulton} to compute its class in $N^*(P, \mathbb{C})$ and  $H^*(P, \mathbb{C})$. If this class is nonzero, our locus must be non-empty. 

 The formula for this class is given in \cite[Theorem 6.13]{Pragacz-sym}. In particular, to obtain the class $[V_\eta^r(f)]$ in $N^*(P, \mathbb{C})$ it is sufficient to compute the Chern classes $c(\mathcal{W}^\vee)$ and $c(\mathcal{U})$. These Chern classes are computed identically as in the unramified case, see \cite[Lemma 5]{DeConciniPragacz} and we have: 
 


 \begin{lm} \label{lm: total Chern} Let $\theta'$ be the restriction to $P$ of the class $\theta$ of the theta divisor on $\textrm{Pic}^{2g-1}(\widetilde{C})$. Then we have 
 	\begin{enumerate}
 		\item $c_i(\mathcal{W}^\vee) = \frac{(\theta')^i}{i!}$ for every $0\leq i \leq g$ and 
 		\item $c_i(\mathcal{U}) = 0$ for every $i > 0$.  
 	\end{enumerate}
  \end{lm} 

Using Theorem \ref{trm: twist-dim}, Lemma \ref{lm: total Chern} and \cite[Theorem 6.13]{Pragacz-sym}, we immediately obtain the formula for $[V_\eta^r(f)]$. 
 \begin{trm} Let $f\colon \widetilde{C}\rightarrow C$ a generic element in $\cR_{g,2}$ and $\eta$ the torsion line bundle defining the double cover. The class of $V^r_\eta(f)$ in $N^*(P, \mathbb{C})$ (or $H^*(P, \mathbb{C})$) is given by the formula 
 	\[ [V^r_\eta(f)] = \prod_{i=1}^{r+1}\frac{i!}{(2i)!}\cdot (\theta')^{\frac{(r+1)(r+2)}{2}}\]
 	In particular, the locus $V^r_\eta(f)$ is non-empty when $g-\frac{(r+1)(r+2)}{2} \geq 0$. 
 \end{trm}
The coefficient on the right-hand side is computed as in \cite[Proposition 6]{DeConciniPragacz}.  
 
 When $g = \frac{(r+1)(r+2)}{2}$ the space $V^r_\eta(f)$ consists of finitely many points. We can use the theorem above to compute its cardinality.
 \begin{trm}
 	Let $g = \frac{(r+1)(r+2)}{2}$, $f\colon \widetilde{C} \rightarrow C$ generic in $\cR_{g,2}$ and $\eta$ its associated torsion line bundle. Then the locus $V^r_\eta(f)$ has cardinality
 	\[  \prod_{i=1}^{r+1}\frac{i!}{(2i)!}\cdot (\theta')^{\frac{(r+1)(r+2)}{2}} = 2^g\cdot g! \cdot \prod_{i=1}^{r+1}\frac{i!}{(2i)!}. \]
 \end{trm} 
 
This approach can be extended completely analogously to the case when $f \colon\widetilde{C} \rightarrow C$ is a generic element of $\cR_{g,4}$.

\begin{trm} \label{twist-dim} Let $f\colon \widetilde{C}\rightarrow C$ a generic element in $\cR_{g,4}$ and $\eta$ the associated torsion line bundle on $C$. Then the twisted Prym-Brill-Noether locus $V^r_\eta(f)$ has dimension $g+1-\frac{(r+1)(r+2)}{2}$ and is smooth away from $V^{r+1}_\eta(f)$. 
	
	The class of $V^r_\eta(f)$ in $N^*(P, \mathbb{C})$ \normalfont{(or $H^*(P, \mathbb{C})$)} is given by the formula 
	 \[ [V^r_\eta(f)] = \prod_{i=1}^{r+1}\frac{i!}{(2i)!}\cdot (\theta')^{\frac{(r+1)(r+2)}{2}}\]
	and when $g + 1 = \frac{(r+1)(r+2)}{2}$ the Prym-Brill-Noether locus $V^r_\eta(f)$ consists of 
	\[	\prod_{i=1}^{r+1}\frac{i!}{(2i)!}\cdot (\theta')^{g+1} = (g+1)!\cdot 2^{g}\cdot \prod_{i=1}^{r+1}\frac{i!}{(2i)!} \]
	reduced points.
\end{trm}
 
 In fact, using the coupled Gieseker-Petri generality appearing in Theorem \ref{trm: coupled k = 1} and Theorem \ref{trm: coupled k = 2}, we can show that 	\[\dim \mathrm{V}_\eta^{\textbf{a}}(f, p) =g+ k - r- 2-\sum_{i=0}^ra_i\]
 when $f\colon \widetilde{C}\rightarrow C$ is a generic double cover in $\cR_{g,2k}$ for $k =1$ or $2$. Moreover, because this twisted Prym-Brill-Noether locus is described as a Lagrangian degeneration, we can use \cite[Type C: symplectic bundles, Corollary]{Anderson-Fulton}  to compute its class. 
 
 \begin{trm} \label{trm: pointed-twist-dim} 
 	Let $f \colon \widetilde{C} \rightarrow C$ be a generic element of $\cR_{g,2k}$ for $k = 1$ or $2$. Moreover, let 
 	\[ \textbf{a} = (0\leq a_0 < a_1 < \cdots < a_r\leq 2g-2+k)\] 
 	be a vanishing sequence and denote $|\textbf{a}| \coloneqq \sum_{i=0}^r a_i$. Then the locus 
 	\[
 	\mathrm{V}_\eta^{\textbf{a}}(f, p) := \left\{ L\in \mathrm{Pic}^{2g-2+k}(\widetilde{C}) \, \left\vert \, 
 	\begin{array}{l}
 		\mathrm{Nm}(L)= \omega_C\otimes \eta,  \\[0.2cm]
 		h^0\text{\large\normalfont(}\widetilde{C}, L(-a_i\,p)\text{\large\normalfont)}\geq r+1-i, \,\, \forall \, i 
 	\end{array}
 	\right.
 	\right\}.
 	\]
 	has dimension $g + k - r- 2 - |\textbf{a}|$. Moreover, its class in $N^*(P, \mathbb{C})$ \normalfont{(or $H^*(P, \mathbb{C})$)} is given by the formula 
 	\[  \prod_{i=0}^r \frac{1}{(a_i+1)!}\prod_{0\leq j < i \leq r} \frac{a_i-a_j}{a_i+a_j+2} (\theta')^{|\textbf{a}| + r +1}\]
 \end{trm}
\begin{proof} For the dimension count, we consider a line bundle $[L]\in V^\textbf{a}_\eta(f,p)$ satisfying $h^0(\widetilde{C}, L-a_ip) = r+1-i$ for every $0\leq i\leq r$. We prove that the tangent space of $V^\textbf{a}_\eta(f,p)$ at this point $L$ has dimension $g + k - r- 2 - |\textbf{a}|$. This tangent space has a similar description to the one in \cite{tarasca-pointed-prym}.

We consider some sections 
\begin{align*}
	&\sigma_i \in H^0\text{\large(} \widetilde{C}, L(-a_i \, p))\setminus H^0( \widetilde{C}, L(-a_{i+1} \, p)\text{\large)}  \quad\mbox{ for } i=0,\dots, r-1, \\
	&\sigma_r \in H^0\text{\large(} \widetilde{C},  L(-a_r \, p)\text{\large)}.
\end{align*}
and the morphism 
\[
\mu \colon \bigoplus_{i=0}^r \, \langle \sigma_i \rangle  \otimes H^0\text{\large(}\widetilde{C}, \iota^*L(a_i \, p)\text{\large)} \rightarrow H^0(\widetilde{C}, \omega_{\widetilde{C}}).
\]

Splitting this morphism into its invariant and anti-invariant parts we obtain the map
\[
\overline{\mu}\colon \bigoplus_{i=0}^r  \,\langle \sigma_i \rangle  \otimes H^0\text{\large(}\widetilde{C}, \iota^*L(a_i \, p)\text{\large)} \Big/ \iota^* H^0\text{\large(} \widetilde{C}, L(-a_{i+1} \, p)\text{\large)} \longrightarrow  H^0(C, \omega_C \otimes \eta).
\]
This anti-invariant map determines the tangent space of $V^\textbf{a}_\eta(f,p)$ at the point $L$. The space $T_L(V_\eta^\textbf{a}(f,p))$ is identified with $\textrm{Im}(\overline{\mu})^\perp$. 

Because $\widetilde{C}$ satisfies the coupled Gieseker-Petri condition, it follows that $\mu$ is injective. This immediately implies that $\overline{\mu}$ is injective. In particular, the tangent space at $[L]$ has dimension 
\[ g+k-1 - \sum_{i=0}^{r} \left(h^0\text{\large(}\widetilde{C}, \iota^*L(a_i \, p)\text{\large)} - h^0\text{\large(} \widetilde{C}, L(-a_{i+1} \, p)\text{\large)}\right). \]
Via the Riemann-Roch theorem, this dimension is equal to   
\[ g+k-1 - \sum_{i=0}^{r} \left(a_i+r+1-i - r+i \right) = g+ k - r- 2-\sum_{i=0}^ra_i. \]
This implies 
\[ \dim \mathrm{V}_\eta^{\textbf{a}}(f, p) \leq g+ k - r- 2-\sum_{i=0}^ra_i. \]

For the reverse inequality, we realize this locus as a Lagrangian degeneracy locus. In the notation of Subsection \ref{sub: Lagrangian}, we have 
	\[ 	\mathrm{V}_\eta^{\textbf{a}}(f, p) = \left\{L \in P \ | \ \dim(\mathcal{U}\cap\mathcal{W}_i)_{L} \geq r+1-i \ \forall\  0\leq i \leq r\right\}. \]
	This locus is a degeneracy locus of Type C, and its dimension is at least $g+k-r-2-\sum_{i=0}^{r}a_i$. The formula for the class $[\mathrm{V}_\eta^{\textbf{a}}(f, p)]$ is an immediate consequence of \cite[Type C: symplectic bundles, Corollary]{Anderson-Fulton}.
\end{proof}
 
 Lastly, we consider the unramified case $k=0$. Let $f\colon \widetilde{C} \rightarrow C$ be a generic \'etale double cover in $\cR_g$ and let $\eta$ be its associated torsion line bundle. We consider the twisted Prym-Brill-Noether loci 
 \[ V^r_\eta(f) \coloneqq \left\{L\in \mathrm{Pic}^{2g-2}(\widetilde{C})\ | \ \textrm{Nm}_f(L) = \omega_C\otimes \eta, \ \textrm{and} \ h^0(\widetilde{C}, L) \geq r+1 \right\}. \]

\begin{trm} \label{trm: unramified-twist}
	Let $f\colon\widetilde{C}\rightarrow C $ a generic element of $\cR_g$ and let $\eta$ its associated torsion line bundle. Then the locus $ V^r_\eta(f)$ has dimension $g-1-\frac{(r+1)(r+2)}{2}$. 
\end{trm}
\begin{proof}
	We know from \cite{Kanev} that 
	\[ \dim V^r_\eta(f) \geq g-1-\frac{(r+1)(r+2)}{2}. \]
	For the reverse inequality, we consider a one dimensional family of double covers 
	\[ \widetilde{\mathcal{C}} \rightarrow \mathcal{C} \rightarrow \Delta \]
	with central fiber $f_0\colon \widetilde{C}_0 \rightarrow C_0$ a generic element of $\Delta^{\textrm{ram}}_0$ (i.e. the normalization of this map is an element of $\cR_{g-1,2}$). 
	
    Let $\mathcal{P} \rightarrow \mathcal{C}$ be the $2$-torsion line bundle corresponding to this family and let $\mathcal{L}^* \rightarrow \widetilde{\mathcal{C}}^*$ be a line bundle satisfying $\textrm{Nm}(\mathcal{L^*})\cong \omega\otimes \mathcal{P}$. 
	When we restrict $\mathcal{P}$ to the central fiber, we obtain a line bundle that has degree $-1$ on the genus $g-1$ component $X$ and degree $1$ on the rational component $R$. Moreover, we have 
	\[ \mathcal{P}_{|X}^{\otimes 2} \cong \OO_X(-x-y)  \]
	where $x, y$ are the nodes connecting the components $X$ and $R$ of the central fiber. 
	The line bundle $\omega \otimes \mathcal{P}$ has degree $2g-3$ on $X$ and degree $1$ on $R$. 
	
	Let $\mathcal{L}$ be an extension of $\mathcal{L}^*$ to $\widetilde{\mathcal{C}}$. Then, the norm condition determines the degree distribution on the central fiber. We have 
	\[ \deg\mathcal{L}_{|\widetilde{X}} = 2g-3 \ \textrm{and} \  \deg\mathcal{L}_{|\widetilde{R}} = 1 \]
	
	In particular, the only line bundles appearing as degenerations of twisted Prym linear series over a generic element of $\Delta_0^{\textrm{ram}}$ satisfy 
	\[ \textrm{Nm}(L) = \omega_X\otimes \eta \ \textrm{and} \ h^0(\widetilde{X}, L) \geq r+1 \]
	This locus has dimension equal to $g-1-\frac{(r+1)(r+2)}{2}$, see Theorem \ref{twist-dim}, and hence it follows that for a generic $f\colon \widetilde{C}\rightarrow C$ in $\cR_g$ we have
		\[ \dim V^r_\eta(f) \leq g-1-\frac{(r+1)(r+2)}{2}. \] 
\end{proof}

In fact, the same method can be used to prove the pointed version of this result. The following result is an immediate consequence of Theorem \ref{trm: pointed-twist-dim} and of the proof of Theorem \ref{trm: unramified-twist}.

 \begin{trm}
	Let $f \colon \widetilde{C} \rightarrow C$ be a generic \'etale double cover in $\cR_{g}$ and $\eta$ its associated torsion line bundle. Moreover, let 
	\[ \textbf{a} = (0\leq a_0 < a_1 < \cdots < a_r\leq 2g-2+k)\] 
	be a vanishing sequence and denote $|\textbf{a}| \coloneqq \sum_{i=0}^r a_i$. Then the locus $\mathrm{V}_\eta^{\textbf{a}}(f, p)$ has dimension $g - r- 2 - |\textbf{a}|$.
\end{trm}

We are left to discuss what happens when the expected dimension is negative. In this case, the coupled Gieseker-Petri condition implies that the pointed twisted Prym-Brill-Noether loci are empty. 

 \begin{trm}
	Let $f \colon \widetilde{C} \rightarrow C$ be a generic element of $\cR_{g,2k}$ for $0\leq k \leq 2$. Moreover, let 
	\[ \textbf{a} = (0\leq a_0 < a_1 < \cdots < a_r\leq 2g-2+k)\] 
	be a vanishing sequence, denote $|\textbf{a}| \coloneqq \sum_{i=0}^r a_i$ and assume $g+k-r-2-|\textbf{a}| < 0$. Then the locus $\mathrm{V}_\eta^{\textbf{a}}(f, p)$ is empty.
\end{trm}
\begin{proof} We will start with the cases $k=1$ and $k=2$. We assume there is a non-empty space of negative expected dimension
		\[
	\mathrm{V}_\eta^{\textbf{a}, \circ}(f, p) := \left\{ L\in \mathrm{Pic}^{2g-2+k}(\widetilde{C}) \, \left\vert \, 
	\begin{array}{l}
		\mathrm{Nm}(L)= \omega_C\otimes \eta,  \\[0.2cm]
		h^0\text{\large(}\widetilde{C}, L(-a_i\,p)\text{\large)} = r+1-i, \,\, \forall \, 0\leq i \leq r 
	\end{array}
	\right.
	\right\}.
	\]
We consider some sections
\begin{align*}
	&\sigma_i \in H^0\text{\large(} \widetilde{C}, L(-a_i \, p)\text{\large)}\setminus H^0\text{\large(} \widetilde{C}, L(-a_{i+1} \, p)\text{\large)}  \quad\mbox{ for } i=0,\dots, r-1, \\
	&\sigma_r \in H^0\text{\large(} \widetilde{C},  L(-a_r \, p)\text{\large)}.
\end{align*}
as well as the anti-invariant part of the pointed Gieseker-Petri morphism
\[
\overline{\mu}\colon \bigoplus_{i=0}^r  \,\langle \sigma_i \rangle  \otimes H^0\text{\large(}\widetilde{C}, \iota^*L(a_i \, p)\text{\large)} \Big/ \iota^* H^0\text{\large(} \widetilde{C}, L(-a_{i+1} \, p)\text{\large)} \longrightarrow  H^0(C, \omega_C \otimes \eta).
\]
Because $[\widetilde{C},p]$ satisfies the coupled Gieseker-Petri condition, it follows that $\overline{\mu}$ is injective. However, the term on the left has dimension $|\textbf{a}| +r+1$ while the one on the right has dimension $g+k-1$. Under the assumption $g+k-r-2-|\textbf{a}| < 0$, injectivity would be impossible. In particular, this implies the theorem for $k=1$ and $k=2$. 

For the case $k=0$, we reason as in Theorem \ref{trm: unramified-twist} to obtain the conclusion. 
\end{proof}
\section{Prym-Brill-Noether loci twisted by an effective divisor}

In this section, we build on \cite{tarasca-pointed-prym} and Section \ref{sec: twist} in order to extend the main results of \cite{Kanev} about the dimension of Prym-Brill-Noether loci twisted by an effective divisor. 

Let $f\colon \widetilde{C} \rightarrow C$ be a generic ramified double cover in $\cR_{g,2k}$ and let $\eta$ be its associated two torsion line bundle. We consider $D$ a degree $d$ effective divisor on $C$ and define the loci 
\[ V^r(f,D) \coloneqq \left\{ L\in \textrm{Pic}^{2g-2-d}(\widetilde{C}) \ | \ \textrm{Nm}_\pi(L) = \omega_C(-D) \ \textrm{and} \ h^0(L) \geq r+1 \right\} \ \textrm{and} \]
\[ V_\eta^r(f,D) \coloneqq \left\{ L\in \textrm{Pic}^{2g-2-d+k}(\widetilde{C}) \ | \ \textrm{Nm}_\pi(L) = \omega_C\otimes \eta (-D) \ \textrm{and} \ h^0(L) \geq r+1 \right\}. \]

We know from \cite[Proposition 1.6]{Kanev} the following inequalities for the dimension of these spaces: 
\begin{itemize}
	\item \[\dim V^r(f,D) \geq g-1+k -(d+k)(r+1)-\frac{r(r+1)}{2} \ \textrm{and}\]
	\item \[\dim V_\eta^r(f,D) \geq g-1+k- d(r+1) - \frac{(r+1)(r+2)}{2}.\]
\end{itemize} 
if they are nonempty.

Our next goal is to show the converse inequality for small values of $k$. These results will be immediate corollaries of \cite[Theorem 2]{tarasca-pointed-prym} and the theorems of Section \ref{sec: twist}. 

\begin{trm} \label{trm: kanev dim}
	Let $f\colon \widetilde{C}\rightarrow C$ be a generic element of $\cR_{g}$ and let $D \in C_d$ be a generic degree $d$  effective divisor on $C$. Then all components of $V^r(f,D)$ have dimension $g-1-\frac{r(r+1)}{2} - d(r+1)$. 
	
	When the expected dimension is negative, the locus $V^r(f,D)$ is empty. 
\end{trm}

\begin{proof} We know from \cite{Kanev} that all components have dimension greater or equal to this. It is sufficient to show the converse inequality. 

We consider the universal Prym-Brill-Noether locus 
\[\mathcal{V}^{r,d}\coloneqq \left\{[f,D,L] \ | \ f \in \cR_g, \ D \in C^d \ \textrm{and} \ L \in V^r(f,D) \right\} \]
and consider the forgetful map to $\mathcal{C}^d\mathcal{R}_g \coloneqq \overline{\mathcal{M}}_{g,d}\times_{\overline{\mathcal{M}_g}}\overline{\mathcal{R}}_g$. Over an element  $[f, D]  \in \mathcal{C}^d\mathcal{R}_g$, the fiber of the forgetful map is the locus $V^r(f, D)$. In particular, it is sufficient to find a unique pair $[f,D]$ such that all the components of $V^r(f, D)$ have the expected dimension. It will then follow that the generic fiber of 
	\[\mathcal{V}^{r,d} \rightarrow \mathcal{C}^d\mathcal{R}_g\]
	has the expected dimension $g-1-\frac{r(r+1)}{2} - d(r+1)$. 
	
	Let $p$ be a generic point on $C$ and let $\widetilde{p}$ be one of the points in the preimage $f^{-1}(p)$. We will show that for $D = d\cdot p$, the fiber $V^r(f, d\cdot p)$ has dimension $g-1-\frac{r(r+1)}{2} - d(r+1)$ as required. 
	
	We consider the map 
	\[ V^r(f, d\cdot p) \xrightarrow{+d\cdot\widetilde{p}} V^r(C,\eta) \amalg V^{r+1}(C,\eta) \]
	There are two obvious loci in the image, namely 
	\begin{enumerate}
		\item The locus $V^\textbf{a}(f,\widetilde{p})$ where $ \textbf{a} = (d, d+1,\ldots, d+r)$ is the vanishing sequence at $\widetilde{p}$; when the generic element in the image has exactly $r+1$ independent global sections, or 
		\item The locus $V^\textbf{a}(f,\widetilde{p})$ where $ \textbf{a} = (0, d, d+1,\ldots, d+r)$; when the generic element in the image has exactly $r+2$ independent global sections.
	\end{enumerate}
	According to \cite[Theorem 2]{tarasca-pointed-prym}, both loci have dimension $g-1-\frac{r(r+1)}{2} - d(r+1)$. 
	
	We show these are the only loci in the image. Assume there exists a component of $V^r(f, d\cdot p)$ such that a generic element $L$ satisfies 
	\[h^0\text{\large(}\widetilde{C}, L(d\cdot \widetilde{p})\text{\large)} = r+1+s \geq r+3.\]
	
	Because $h^0(\widetilde{C}, L) \geq r+1$, the vanishing orders of $L(d\widetilde{p})$ at $\widetilde{p}$ are $\geq (0,1,\ldots, s-1, d, d+1, \ldots, d+r)$.
	According to \cite[Theorem 2]{tarasca-pointed-prym}, the locus respecting these properties has dimension \[g-1-\frac{r(r+1)}{2} - d(r+1) - \frac{s(s-1)}{2}.\] 
	This is not possible as the map $+d\cdot \widetilde{p}$ is injective. 
	
	In conclusion, all components of the space $V^r(f, d\cdot p)$ have dimension 
	\[ g-1-\frac{r(r+1)}{2}-d(r+1).\]
\end{proof}

The same method as in Theorem \ref{trm: kanev dim} can be used to estimate the dimension of twisted Prym-Brill-Noether loci. Using the results of the previous section, we obtain: 

\begin{trm} \label{trm: kanev twist dim}
	Let $f\colon \widetilde{C}\rightarrow C$ be a generic element of $\cR_{g,2k}$ for $0\leq k\leq 2$ and let $D \in C_d$ be a generic degree $d$  effective divisor on $C$. Then the locus $V_\eta^r(f,D)$ has dimension $g-1 + k-\frac{(r+1)(r+2)}{2} - d(r+1)$. When the expected dimension is negative, the locus $V_\eta^r(f,D)$ is empty. 
\end{trm}
\section{Prym-Brill-Noether and pointed Brill-Noether conditions}

We consider the universal Prym-Brill-Noether locus $\mathcal{V}^r_g$ defined as 
\[ \mathcal{V}^r_g \coloneqq \left\{[C,\eta, L] \ | \ [C,\eta] \in \mathcal{R}_g \ \textrm{and} \ L\in V^r(C,\eta) \right\}. \]
This locus fits into a diagram 
\[
\begin{tikzcd}
	\mathcal{V}^r_g  \arrow{r}{} \arrow[swap]{d}{} & ? \arrow{d}{} \\
	\cR_g \arrow{r}{i}&  \rr_g
\end{tikzcd}
\]
and our goal is to construct a partial compactification of it over $\rr_g$. To do this, we first remark that a generic element of $\mathcal{V}^r_g$ can be viewed as a linear series $g^r_{2g-2}$ satisfying the norm condition. As in \cite{BudPrymIrr}, we will call these Prym limit $g^r_{2g-2}$. 

Let $\widetilde{\cR}_g$ be the partial compactification of $\cR_g$ consisting of double covers $f\colon \widetilde{C} \rightarrow C$ whose source curve is of compact type. We consider the map 
\[ \chi_g\colon \widetilde{\cR}_g \rightarrow \cM^{\textrm{ct}}_{2g-1},\]
along with the moduli space $\mathcal{G}^r_{2g-2}(\cM^{\textrm{ct}}_{2g-1})$ parametrizing crude limit $g^r_{2g-2}$ on curves in $\cM^{\textrm{ct}}_{2g-1}$.

In this setting, we have a diagram
\[
\begin{tikzcd}
	\mathcal{V}^r_g  \arrow{r}{} \arrow[swap]{d}{} & \widetilde{\cR}_g\times_{\cM^{\textrm{ct}}_{2g-1}}\mathcal{G}^r_{2g-2}(\cM^{\textrm{ct}}_{2g-1}) \arrow{d}{} \\
	\cR_g \arrow{r}{i}&  \widetilde{\cR}_g
\end{tikzcd}
\]
and our goal is to understand the closure of $\mathcal{V}^r_g $ in $\widetilde{\cR}_g\times_{\cM^{\textrm{ct}}_{2g-1}}\mathcal{G}^r_{2g-2}(\cM^{\textrm{ct}}_{2g-1}) $. This closure is contained in the locus of crude limit linear series respecting the norm condition, which we will call Prym limit $g^r_{2g-2}$'s. 

Let $[f\colon \widetilde{C}\rightarrow C] \in \rr_g$ such that $C$ is a curve of compact type admitting a unique irreducible component $X$ for which $\eta_X \ncong \OO_X$. For this component $X$, we denote by $p^X_1,\ldots, p^X_{s_X}$ its nodes and by $g^X_1,\ldots, g^X_{s_X}$ the genera of the connected components of $C\setminus X$ glued to $X$ at these points.  For an irreducible component $Y$ of $C$, different from $X$, we denote by $q^Y$ the node glueing $Y$ to the connected component of $C\setminus Y$ containing $X$, and by $p^Y_1,\ldots, p_{s_Y}^Y$ the other nodes of $Y$. We denote by $g^Y_0, g^Y_1,\ldots, g^Y_{s_Y}$ the genera of the connected components of $C\setminus Y$ glued to $Y$ at these points. Using the above notations, we can define the concept of a Prym limit $g^r_{2g-2}$: 
\begin{defi} \label{def: Prym-limit} \cite{BudPrymIrr}
	Let $[f\colon \widetilde{C}\rightarrow C] \in \rr_g$ as above. A Prym limit $g^r_{2g-2}$ on $f$, denoted $\ell$, is a crude limit $g^r_{2g-2}$ on $\widetilde{C}$ satisfying the following two conditions: 
	\begin{enumerate}
		\item For the unique component $\widetilde{X}$ of $\widetilde{C}$ above $X$, the $\widetilde{X}$-aspect $L_{\widetilde{X}}$ of $\ell$ satisfies
		\[Nm_{f_{|\widetilde{X}}} L_{\widetilde{X}} \cong \omega_X(\sum_{i=1}^s2g^X_ip_i)\]
		\item For a component $Y$ of $C$ different from $X$, we denote by $Y_1$ and $Y_2$ the two irreducible components of $\widetilde{C}$ above it. We identify these two components with $Y$ via the map $f$. With this identification the $Y_1$ and $Y_2$ aspects of $\ell$ satisfy: 
		\[ L_{Y_1}\otimes L_{Y_2} \cong \omega_Y\text{\large(}(2g-2+2g^Y_0)q^Y + \sum_{i=1}^sg_i^Yp_i^Y\text{\large)} \]
	\end{enumerate} 
	 We denote by $PG^r_{2g-2}(f)$ the space of Prym limit $g^r_{2g-2}$ on $f$. Furthermore, we denote by $V^r(f) \subseteq PG^r_{2g-2}(f)$ to be the set of limit linear series appearing in the closure of $\mathcal{V}^r_g$ above $[f] \in \widetilde{\cR}_g$.  
\end{defi}

Let $f\colon Y_1\cup_{x_1}\widetilde{E}\cup_{x_2}Y_2\rightarrow Y\cup_xE$ be the double cover corresponding to a generic element in the boundary component $\Delta_1\subseteq \rr_g$. We consider the map 
\[ V^r(f) \rightarrow \textrm{Pic}^{2g-2}(Y_1)\times \textrm{Pic}^{2g-2}(\widetilde{E})\times \textrm{Pic}^{2g-2}(Y_2)\]
forgetting the $(r+1)$-dimensional spaces of sections and only remembering the underlying line bundles. We will show that the image of this map has dimension $g-1-\frac{r(r+1)}{2}$. 

\begin{prop} \label{Prym limits} Let  $f\colon Y_1\cup_{x_1}\widetilde{E}\cup_{x_2}Y_2\rightarrow Y\cup_xE$ be a generic double cover in the boundary component $\Delta_1\subseteq \rr_g$, and let $\left\{l_1 \coloneqq (V_1,L_1), (V_{\widetilde{E}}, L_{\widetilde{E}}), l_2\coloneqq (V_2,L_2) \right\}$ be a generic Prym limit $g^r_{2g-2}$ in $V^r(f)$. Then we have that 
	\begin{itemize}
		\item The vanishing orders of the $Y_1$-aspect at $x_1$ are $(g-r-1, g-r+1, \ldots, g+r-1)$,
		\item The vanishing orders of the $Y_2$-aspect at $x_2$ are $(g-r-1, g-r+1, \ldots, g+r-1)$ and 
		\item The vanishing orders of the $\widetilde{E}$-aspect at $x_1$ and $x_2$ are both $(g-r-1, g-r+1, \ldots, g+r-1)$.
	\end{itemize}
\end{prop}
\begin{proof} We first interpret the norm conditions that this limit linear series satisfy.
\begin{enumerate} \item When we identify $Y_1, Y_2$ with $Y$ we get the isomorphism
		\[ L_1\otimes L_2 \cong \omega_Y\text{\large(}2g\cdot x\text{\large)} \]
		\item For the $\widetilde{E}$-aspect we have 
		\[ \textrm{Nm}(L_{\widetilde{E}}) = \omega_E\text{\large(}(2g-2)\cdot x\text{\large)}. \]
\end{enumerate}

In particular knowing $L_1$ or $L_2$ uniquely determines the other. Second, $L_{\widetilde{E}}$ is isomorphic to either $\OO_{\widetilde{E}}\text{\large(}(2g-2)\cdot x_1\text{\large)}$ or $\OO_{\widetilde{E}}\text{\large(}(2g-3)\cdot x_1 + x_2\text{\large)}$.  

The generic fiber of the morphism $\mathcal{V}^r_g\rightarrow \cR_g$ has dimension $g-1-\frac{r(r+1)}{2}$. This implies that the dimension of (all components of) $V^r(f)$ is greater or equal to  $s\coloneqq g-1-\frac{r(r+1)}{2}$.


We look at the image of the map 
\[ V^r(f) \rightarrow \textrm{Pic}^{2g-2}(Y_1)\times \textrm{Pic}^{2g-2}(\widetilde{E})\times \textrm{Pic}^{2g-2}(Y_2).\]
Because the image is at least $s$-dimensional, it implies that 
\[ \rho(l_1, x_1) \geq s \ \textrm{and} \  \rho(l_2, x_2) \geq s. \]
Using Brill-Noether additivity, see \cite[Proposition 4.6]{limitlinearbasic} and the inequality $\rho(\widetilde{E}, x_1,x_2)\geq -r$ which is always satisfied for linear series on elliptic curves, see \cite[Proposition 1.4.1]{FarkasThesis} we obtain the inequalities 
\[ \rho(2g-1,r,2g-2) = -r + 2s \geq \rho(l_1,x_1) + \rho(\widetilde{E}, x_1,x_2) + \rho(l_2,x_2) \geq -r +2s. \]
This implies that 
\[ \rho(l_1, x_1) = s, \   \rho(l_2, x_2) = s \ \textrm{and} \ \rho(\widetilde{E}, x_1,x_2) = -r.  \]
We denote by $0\leq a_0 < a_1<\cdots <a_r \leq 2g-2$ and $0\leq b_0 < b_1<\cdots <b_r \leq 2g-2$ the vanishing orders at $x_1, x_2$ for the $Y_1$ and $Y_2$ aspects respectively. The equality $\rho(\widetilde{E}, x_1,x_2) = -r$ implies that $a_i + b_{r-i} = 2g-2$ for all $0\leq i\leq r$.  
Reasoning as in \cite{BudPrymIrr} and \cite{BudPBN} we obtain that: 
\begin{itemize}
	\item All the $a_i$'s have the same parity: otherwise using the equivalence of divisors $2x_1 \equiv 2x_2$ on $\widetilde{E}$ we obtain the contradiction $x_1 \equiv x_2$. In particular, this implies $a_i \geq a_{i-1}+2$ for each $1\leq i \leq r$. 
	\item $a_0 = g-r-1$ and $a_r = g+r-1$:
	
	The genericity of $[Y_2,x_2] \in \cM_{g-1,1}$ together with $ \rho(L_2,x_2) = s$ imply that 
	$$h^0(Y_2, L_2(-b_ix_2)) = r+1-i \ \textrm{for all} \ 0\leq i \leq r.$$ 
	 Using that $L_1\otimes L_2 \cong \omega_Y(2g\cdot x)$ and the Riemann-Roch theorem we obtain 
	\[ h^0\text{\large(}Y_1, L_1(-(2+a_{r-i})q)\text{\large)} = g+r-1-a_{r-i}-i\]  
	Choosing $i=0$ we get $a_r = g+r-1$.  Inverting the roles of the $a_i$'s and $b_i$'s we obtain that $a_0 = g-r-1$.
\end{itemize}
These two conditions immediately imply $a_i = g-r-1+2i$ for each $0\leq i \leq r$. Inverting the roles of the $a_i$'s and the $b_i$'s we conclude
\[ a_i = b_i = g-r-1 + 2i \ \textrm{for every} \ 0\leq i \leq r. \]
\end{proof}

In particular, we recover the dimension count for $V^r(f)$ for a generic $f\colon \widetilde{C} \rightarrow C$ in $\cR_{g}$. 
\begin{cor} \label{cor: dimPBN}
 Let $f\colon \widetilde{C} \rightarrow C$ be a generic double cover in  $\cR_{g}$. Then we have 
 \[ \dim V^r(f) = g-1-\frac{r(r+1)}{2}.\]
 Moreover, $V^r(f)$ is empty when $g-1-\frac{r(r+1)}{2} < 0$. 
\end{cor}
\begin{proof}
	Consider the partial compactification $\overline{\mathcal{V}}^r_g$ over the boundary divisor $\Delta_1$ and consider the forgetful morphism 
	\[ \overline{\mathcal{V}}^r_g \rightarrow \rr_g.\]
	Using the previous proposition, we obtain that the fibre over a generic double cover in $\Delta_1$ has dimension $g-1-\frac{r(r+1)}{2}$. It implies that the generic fiber satisfies 
	\[ \dim V^r(f) \leq g-1-\frac{r(r+1)}{2}. \]
	Realizing $V^r(f)$ as a Lagrangian degeneracy locus, we obtain the converse inequality 
    \[ \dim V^r(f) \geq g-1-\frac{r(r+1)}{2}. \]
	This concludes the proof.
\end{proof}

Let $ \textbf{a} = (0 \leq a_0 < a_1 < \cdots < a_r \leq 2g-2)$ be the sequence of vanishing orders defined by $a_i = 2i$. We proved in Proposition \ref{Prym limits} that over a generic element $[Y\cup_x E, \OO_Y,E]$ of $\Delta_1\subseteq \rr_g$ the locus of smoothable Prym limit linear series is (after removing the base-locus) isomorphic to 
 \[ W^r_{g+r-1,\textbf{a}}(Y)\coloneqq \left\{L\in \textrm{Pic}^{g+r-1}(Y) \ | \ h^0(Y, L-2i\cdot x) \geq r+1-i \ \forall \ 0\leq i \leq r \right\}.\]

We can extend these results to the moduli space $\mathcal{R}_{g,2}$. The Prym limit linear series for double covers $[f\colon \widetilde{C}\rightarrow C]$ in $\mathcal{R}_{g,2}$ are described as in Definition \ref{def: Prym-limit} to be limit linear series respecting the norm condition. 

We define the universal Prym-Brill-Noether locus to be 
\[ \mathcal{V}^r_{g,2} \coloneqq \left\{[C,\eta, x+y, L] \ | \ [C,\eta,x+y] \in \mathcal{R}_{g,2} \ \textrm{and} \ L\in V^r(C,\eta, x+y) \right\}. \]
As in the unramified case, this can be fitted into a diagram 
\[
\begin{tikzcd}
	\mathcal{V}^r_{g,2}  \arrow{r}{} \arrow[swap]{d}{} & \widetilde{\cR}_{g,2}\times_{\cM^{\textrm{ct}}_{2g}}\mathcal{G}^r_{2g-2}(\cM^{\textrm{ct}}_{2g}) \arrow{d}{} \\
	\cR_{g,2} \arrow{r}{i}&  \widetilde{\cR}_{g,2}
\end{tikzcd}
\]
and our next goal is to understand what limit linear series appear in the closure of $	\mathcal{V}^r_{g,2} $. We are interested in understanding this closure over a generic element in the boundary divisor $\Delta_{0:g,\left\{\mathcal{O}\right\}} \subseteq \rr_{g,2}$, as defined in \cite[Section 2]{BudKodPrym}. The associated double cover for this element is of the form 
\[ f\colon Y_1\cup_{p_1}\widetilde{R}\cup_{p_2}Y_2 \rightarrow Y\cup_pR  \]
where 
\begin{itemize}
	\item $f\colon \widetilde{R} \rightarrow R$ is a double cover between rational curves, ramified at two points $\widetilde{x}, \widetilde{y} \in \widetilde{R}$. The points $p_1, p_2$ are the points in the preimage of $p$. 
	\item the marked curve $[Y,p]$ is generic in $\cM_{g,1}$, and $[Y_1, p_1]$, $[Y_2, p_2]$ are two copies of it. The map $f$ identifies these copies with $[Y,p]$. 
\end{itemize}

Analogously to Definition \ref{def: Prym-limit}, we can define Prym limit linear series for this double cover $$f\colon Y_1\cup_{p_1}\widetilde{R}\cup_{p_2}Y_2 \rightarrow Y\cup_pR.$$

	\begin{defi} \label{def: Prym-limit ramified}
		Let $f\colon Y_1\cup_{p_1}\widetilde{R}\cup_{p_2}Y_2 \rightarrow Y\cup_pR$ as above. A Prym limit $g^r_{2g-2}$ on $f$, denoted $\ell$, is a limit $g^r_{2g-2}$ on $Y_1\cup_{p_1}\widetilde{R}\cup_{p_2}Y_2$ satisfying the following condition: 
		\begin{enumerate}
			\item  Via the identification of $Y_1$ and $Y_2$ with $Y$, the $Y_1$ and $Y_2$ aspects of $\ell$ satisfy: 
			\[ L_{Y_1}\otimes L_{Y_2} \cong \omega_Y\text{\large(}(2g-2)p \text{\large)} \]
		\end{enumerate} 
		We denote by $PG^r_{2g-2}(f)$ the space of Prym limit $g^r_{2g-2}$ on $f$. Furthermore, we denote by $V^r(f) \subseteq PG^r_{2g-2}(f)$ the subset of elements that appear in the closure of $\mathcal{V}^r_{g,2}$ over $f$. 
	\end{defi}

Our next goal is to describe the elements of $V^r(f)$, as we did in the unramified case in Proposition \ref{Prym limits}. To achieve this, we will use the following observations from Section \ref{sec: PBN ramified}:
\begin{itemize}
	\item For an element $f\colon \widetilde{C} \rightarrow C$ in $\cR_{g,2}$, the following Prym-Brill-Noether loci are identified via Serre duality: 
	
	\[ V^r(f) \rightarrow V^{r+1}(f, x+y), \ \  L \mapsto \iota^*L\]
	
   \item Because all sections in $H^0(C, \omega_C(x))$ vanish at $x$, the space $V^{r+1}(f, x+y)$ is identified with 
   
   \[\left\{ L \in \textrm{Pic}^{2g}(\widetilde{C}) \ | \ \textrm{Nm}_f(L) = \omega_C(x+y), h^0(\widetilde{C}, L) \geq r+2 \ \textrm{and} \  h^0(\widetilde{C}, L(-\widetilde{x}-\widetilde{y})) \geq r+1\right\} \]
   via adding the superfluous condition $ h^0(\widetilde{C}, L(-\widetilde{x}-\widetilde{y})) \geq r+1$. 
\end{itemize}

We define the space 
\[ \mathcal{V}^r_{g,2}(\textbf{x}+\textbf{y}) \coloneqq  \left\{[f\colon \widetilde{C} \rightarrow C, L] \ | \ [f] \in \mathcal{R}_{g,2}, \ \textrm{branched at points} \ x,y \ \textrm{and} \ L \in V^{r}(f, x+y)\right\}. \]
Via Serre duality, we have an isomorphism between $\mathcal{V}^r_{g,2}$ and $\mathcal{V}^{r+1}_{g,2}(\textbf{x}+\textbf{y})$. Hence, in order to understand how $ \mathcal{V}^r_{g,2}$ degenerates to the boundary, it is sufficient to answer this question for the other moduli space. 

We consider again the double cover $f\colon Y_1\cup_{p_1}\widetilde{R}\cup_{p_2}Y_2 \rightarrow Y\cup_pR$ associated to a generic element in the boundary divisor $\Delta_{0:g,\left\{\mathcal{O}\right\}}$.
\begin{defi} \label{def: PG-x+y}
	Let $f\colon Y_1\cup_{p_1}\widetilde{R}\cup_{p_2}Y_2 \rightarrow Y\cup_pR$ as above. We define $PG^r(f, x+y)$ to be the space of limit $g^r_{2g}$ on $Y_1\cup_{p_1}\widetilde{R}\cup_{p_2}Y_2$ satisfying the following conditions: 
\begin{enumerate} 
	\item For the $\widetilde{R}$-aspect $(V_{\widetilde{R}},L_{\widetilde{R}})$ we have 
	\[ h^0\text{\large(}\widetilde{R}, V_{\widetilde{R}}(-\widetilde{x}-\widetilde{y})\text{\large)}\geq r.\]
	\item  Via the identification of $Y_1$ and $Y_2$ with $Y$, the $Y_1$- and $Y_2$- aspects of $L$ satisfy: 
	\[ L_{Y_1}\otimes L_{Y_2} \cong \omega_Y\text{\large(}(2g+2)p \text{\large)}. \]
\end{enumerate}
We denote by $V^r(f,x+y) \subseteq PG^r(f, x+y)$ the set of those linear series that appear in the closure of $\mathcal{V}^r_{g,2}(\textbf{x}+\textbf{y})$ above the boundary divisor $\Delta_{0:g,\left\{\mathcal{O}\right\}}$. 
\end{defi}

In this case, we have an analogue of Proposition \ref{Prym limits}:

\begin{prop} \label{Prym limits ramified} Let $f\colon Y_1\cup_{p_1}\widetilde{R}\cup_{p_2}Y_2\rightarrow Y\cup_pR$ a generic double cover in the boundary component $\Delta_{0:g,\left\{\OO\right\}}\subseteq \rr_{g,2}$ and let $\left\{l_1 \coloneqq (V_1,L_1), \widetilde{l} \coloneqq (V_{\widetilde{R}}, L_{\widetilde{R}}), l_2\coloneqq (V_2,L_2) \right\}$ be a generic limit linear series in $V^r(f,x+y)$. Then we have that 
	\begin{itemize}
		\item The vanishing orders of the $Y_1$-aspect at $p_1$ are $(g-r, g-r+2, \ldots, g+r)$,
		\item The vanishing orders of the $Y_2$-aspect at $p_2$ are $(g-r, g-r+2, \ldots, g+r)$ and 
		\item The vanishing orders of the $\widetilde{R}$-aspect at $p_1$ and $p_2$ are both $(g-r, g-r+2, \ldots, g+r)$.
	\end{itemize}
\end{prop}
\begin{proof}
   We know from Theorem \ref{trm: dim-ramified} that for a generic double cover $g$ in $\cR_{g,2}$ branched over $x'$ and $y'$, the dimension of the locus $V^r(g,x'+y')$ is $g-\frac{r(r+1)}{2}$. This implies the inequality 
   \[ \dim V^r(f,x+y) \geq  s\coloneqq g-\frac{r(r+1)}{2} . \]
   Because we have 
   	\[ L_{Y_1}\otimes L_{Y_2} \cong \omega_Y\text{\large(}(2g+2)p \text{\large)} \]
  it follows that one of the line bundles uniquely determines the other. Combined with the inequality of the dimension, this implies 
  \[ \rho(L_{Y_1}, p_1) \geq s, \ \textrm{and} \ \rho(L_{Y_2}, p_2) \geq s. \]
  We use Brill-Noether additivity to obtain 
  \[ \rho(2g,r,2g) = 2s \geq \rho(L_{Y_1}, p_1) + \rho(L_{Y_2}, p_2) + \rho(\widetilde{l}, p+1,p_2) \geq s + s + 0.  \] 
  In particular, all inequalities are in fact equalities and a generic limit linear series in $V^r(f,x+y)$ is refined.
  
  We denote by $0\leq a_0 < a_1<\cdots <a_r \leq 2g$ the vanishing orders of the $Y_1$-aspect at the point $p_1$. As in Proposition \ref{Prym limits} we have that $a_0 = g-r$ and $a_r = g+r$. 
  
  In order to conclude the proposition, it is sufficient to show $a_{i+1}-a_i \geq 2$ for every $0\leq i \leq r-1$. For this we look at the $\widetilde{R}$-aspect of the limit linear series. Its vanishing orders at $p_1$ are $(2g-a_r, \ldots, 2g-a_0)$ while the vanishing orders at $p_2$ are $(a_0,\ldots, a_r)$. We consider the subspace 
  \[ V_{\widetilde{R}}(-\widetilde{x}-\widetilde{y}) \subseteq V_{\widetilde{R}} \]
  which is of dimension $r$ because of the condition 
  	\[ h^0\text{\large(}\widetilde{R}, V_{\widetilde{R}}(-\widetilde{x}-\widetilde{y})\text{\large)}\geq r\]
  appearing in Definition \ref{def: PG-x+y}.
  
  We look at the vanishing orders at $p_1, p_2$ of the sections in $V_{\widetilde{R}}(-\widetilde{x}-\widetilde{y})\text{\large)}$. These vanishing orders are subsets of $(2g-a_r, \ldots, 2g-a_0)$ and $(a_0,\ldots, a_r)$ missing exactly one element. 
  
  The orders at $p_1$ and those at $p_2$ of $V_{\widetilde{R}}(-\widetilde{x}-\widetilde{y})\text{\large)}$ can be paired so that the sum of each pair is less or equal to $2g-2$. This implies that the vanishing orders at $p_1$ are  $(2g-a_r, \ldots, 2g-a_1)$ and the orders at $p_2$ are $(a_0,\ldots, a_{r-1})$. This furthermore implies that 
  \[ a_i + 2g-a_{i+1} \leq 2g-2 \]
  and hence $a_{i+1}-a_i \geq 2$ for each $0\leq i \leq r-1$. This implies the conclusion.
\end{proof}

Using Serre duality and the previous result, we can describe the linear series appearing as limits of $\mathcal{V}^r_{g,2}$ to the boundary. 

\begin{prop} Let $f\colon Y_1\cup_{p_1}\widetilde{R}\cup_{p_2}Y_2\rightarrow Y\cup_pR$ a generic double cover in the boundary component $\Delta_{0:g,\left\{\OO\right\}}\subseteq \rr_{g,2}$ and let $\left\{l_1 \coloneqq (V_1,L_1), \widetilde{l} \coloneqq (V_{\widetilde{R}}, L_{\widetilde{R}}), l_2\coloneqq (V_2,L_2) \right\}$ be a generic limit linear series in $V^r(f)$. Then we have that 
	\begin{itemize}
		\item The vanishing orders of the $Y_1$-aspect at $p_1$ are $(g-r-1, g-r+1, \ldots, g+r-1)$,
		\item The vanishing orders of the $Y_2$-aspect at $p_2$ are $(g-r-1, g-r+1, \ldots, g+r-1)$ and 
		\item The vanishing orders of the $\widetilde{R}$-aspect at $p_1$ and $p_2$ are both $(g-r-1, g-r+1, \ldots, g+r-1)$.
	\end{itemize}
\end{prop}

\begin{proof}
	The line bundle $\omega_Y(2g\cdot p)\otimes L_1^{-1}$ is the $Y_1$-aspect of a generic element in $V^{r+1}(f,x+y)$. Its vanishing orders at $p_1$ are described in the previous proposition. We have 
	\[ h^0\text{\large(}Y, \omega_Y(2g\cdot p)\otimes L_1^{-1} - (g-r-1+2i)\cdot p\text{\large)} = r+2-i. \]
	Using the Riemann-Roch Theorem we obtain 
	\[ h^0\text{\large(}Y, L_1-(g+r+1-2i)\cdot p\text{\large)} = (r+2-i)-(r+2-2i) = i.\] 
	This immediately implies the conclusion. 
\end{proof}

\section{Du Val curves and Prym-Brill-Noether generality} \label{sec: surfaces}

The results of \cite{Farkas-Tarasca-BNgen} tell us that we can find pointed Du Val curves that are pointed Brill-Noether general. Using Proposition \ref{Prym limits} we can further prove that a generic Du Val curve is Prym-Brill-Noether general.  

\textbf{Proof of Theorem \ref{trm: DuVal}:} Let $J$ be the unique smooth plane cubic passing through the points $p_1, \ldots, p_9$ and denote by $J'$ its strict transform. The elliptic curve $J'$ is contained in the linear system $|3l-E_1-\cdots-E_9|$. 

We consider $D$ to be a generic curve in the linear system 
\[ L_{g-1} \coloneqq |3(g-1)l-(g-1)E_1-\cdots -(g-1)E_8-(g-2)E_9| \]
and we know that $D$ and $J'$ intersect at an unique point $p$. We know from \cite[Theorem 1]{Farkas-Tarasca-BNgen} that $[D,p]$ is pointed Brill-Noether general. 

We look at the curve $D\cup_pJ'$ in $L_g$ and associate to it a Prym curve $[D\cup_pJ', \eta]$ in the boundary divisor $\Delta_1 \subseteq \rr_g$. Using Proposition \ref{Prym limits} and reasoning as in Corollary \ref{cor: dimPBN} we conclude that all Prym-Brill-Noether loci $V^r(D\cup_pJ', \eta)$ are empty or of expected dimension. This implies the conclusion. \hfill $\square$

We can use a similar approach to find Prym-Brill-Noether general curves on elliptic ruled surfaces. We start with an elliptic curve $J$ and a non-torsion line bundle $L$ in $\textrm{Pic}^0(J)$. We consider the decomposable ruled surface 
\[ \varphi\colon Y\coloneqq \mathbf{P}(\OO_J\oplus L) \rightarrow J\]
which comes equipped with two disjoint sections $J_0$ and $J_1$. We choose a point $r \in J$ and consider $f\coloneqq \varphi^{-1}(r)$ the corresponding ruling of $Y$. 

In the situation above, the linear system $|gJ_0+f|$ consists of curves of genus $g$ and moreover, for a suitable choice of a $2$-torsion line bundle, a generic such curve is Prym-Brill-Noether general. 
\begin{prop} \label{prop:ell-fibration} In the setting outlined above, let $C \in \cM_g$ be a generic element of the linear system $|gJ_0+f|$. Then there exists a $2$-torsion line bundle $\eta$ on $C$ such that $[C,\eta]$ is Prym-Brill-Noether general.  
\end{prop}
\begin{proof}
	We consider $D$ to be a generic element in the linear system $|(g-1)J_0 +f|$. Then $D$ is a smooth curve of genus $g-1$, it intersects $J_0$ in a unique point $p$ and moreover $[D,p]$ is pointed Brill-Noether general, see \cite[Theorem 2]{Farkas-Tarasca-BNgen}. 
	
	In particular, we consider a Prym curve $[D\cup_pJ_0, \eta]$ in the boundary divisor $\Delta_1 \subseteq \rr_g$ and apply Proposition \ref{Prym limits} to conclude that all Prym-Brill-Noether loci $V^r(D\cup_pJ_0, \eta)$ are either empty or have the expected dimension. This concludes the proof.
\end{proof} 

The curves appearing in Theorem \ref{trm: DuVal} and in Proposition \ref{prop:ell-fibration} are contained in the closure of the locus of curves lying on $K3$ surfaces. As a consequence, we have the following result:
\begin{cor} Let $(S,H)$ be a smooth polarized $K3$ surface of degree $2g-2$ and let $C$ be a generic curve in the linear system $|H|$. Then there exists a $2$-torsion line bundle $\eta$ on $C$ such that $[C,\eta]$ is Prym-Brill-Noether general.  
\end{cor}

\bibliography{main}
\bibliographystyle{alpha}
\end{document}